\documentclass{amsart}
\usepackage{amsmath}
\usepackage{amssymb}
\usepackage{amsthm}
\usepackage{dsfont}
\usepackage{lineno}
\usepackage{subfigure}
\usepackage{graphicx}
\usepackage{multirow}
\usepackage{color}
\usepackage{xspace}
\usepackage{pdfsync}
\usepackage{hyperref} 
\usepackage{algorithmicx}
\usepackage{algpseudocode}
\usepackage{algorithm}
\usepackage{mathtools}
\usepackage{enumitem}
\usepackage{comment}
\graphicspath{{Figures/}}

\newcommand{\bq}{\begin{equation}}
\newcommand{\eq}{\end{equation}}

\algnewcommand{\LineComment}[1]{\State \(\triangleright\) #1}

\newtheorem{theorem}{Theorem}

\theoremstyle{lemma}
\newtheorem{lemma}[theorem]{Lemma}

\newtheorem{definition}[theorem]{Definition}
\newtheorem{proposition}[theorem]{Proposition}

\newtheorem{remark}[theorem]{Remark}

\newtheorem{hypothesis}[theorem]{Hypothesis}

\theoremstyle{remark}

\makeatletter
\newcommand\appendix@section[1]{%
\refstepcounter{section}%
\orig@section*{Appendix \@Alph\c@section: #1}%
}
\let\orig@section\section
\g@addto@macro\appendix{\let\section\appendix@section}
\makeatother

\begin{document}

\title[Optimal Transport in Euclidean Space with Unbounded Support]{Pointwise Convergence Analysis for Approximations of Optimal Transport Problems with a Target Measure that Has Unbounded Support}

\author{Axel G. R. Turnquist}
\address{Beijing Institute of Mathematical Sciences and Applications, Beijing, China}
\email{agrt@bimsa.cn}

\begin{abstract}
We consider the Monge problem of optimal transport between a compactly supported source measure and a target probability measure with unbounded support. We consider the convergence of optimal maps and potential functions when the target measure is approximated, with special attention given to a cutoff approximation in which we parametrize the approximation by a ``cutoff" radius $R$ for the target measure. We study both the convergence of the mapping and potential functions for the forward and inverse problem in many cases such as 1) the radially symmetric case with the cutoff approximation for general cost functions and 2) the non-radially symmetric case with the squared distance cost function. We derive quantitative non-asymptotic pointwise convergence rates in special cases, building on the $L^2$ convergence rates established in Delalande and M\`{e}rigot~\cite{merigot1}. These results can be used, for instance, to justify the use of certain types of numerical Monge-Amp\`{e}re equation solvers in computationally solving the problem.
\end{abstract}

\date{\today}    
\maketitle
\tableofcontents

\section{Introduction}\label{sec:introduction}
In this manuscript, we consider solving the Monge optimal transport problem in Euclidean space from a source measure with bounded support to a target measure with unbounded support. We derive quantitative pointwise convergence rates of a cutoff approximation of such a problem.

\subsection{The Monge Problem of Optimal Transport}

We define the Monge problem of optimal transport on $\mathbb{R}^n$. A typical presentation of optimal transport will include the Kantorovich formulation, see any book on optimal transport, for example the very popular Villani~\cite{Villani1}, but we choose to omit this for the sake of brevity.

\begin{definition}[Monge problem of optimal transport]
Given probability distributions $\mu, \nu$ (source and target measures, respectively) supported on $\mathbb{R}^n$ and a cost function $c: \mathbb{R}^n \times \mathbb{R}^n \rightarrow \mathbb{R}$, we define the Monge problem of optimal transport to be the minimization of the following cost functional:
\begin{equation}
C(\mu, \nu) := \int_{\mathbb{R}^d} c(x,S(x)) d\mu(x),
\end{equation}
subject to the constraint $\mu(S^{-1}(E)) = \nu(E)$ for every Borel set $E \subset \mathbb{R}^n$, where $S^{-1}(E) := \{x \in \mathbb{R}^d: S(x) \in E \}$, which is often denoted by the shorthand $T_{\#} \mu = \nu$. The optimal map $T$ is often known as a pushfoward map, since $\nu$ is the pushforward measure of $\mu$ with respect to the map $T$.
\end{definition}

Of course, in general, the cost functional does not necessarily have a minimizer. Even if it has a minimizer, it may not be unique. In order to guarantee the existence of minimizers of the Monge problem, for, say the cost function $c(x,y) = \vert x - y \vert^2$, one has to restrict the class of source measures. A simple condition, which is sufficient but not necessary, is that $\mu$ is absolutely continuous with respect to the Lebesgue measure. Then, in this case, not only can be guarantee the existence of minimizers, but the choice of cost function $c(x,y) = \vert x - y \vert^2$ will actually guarantee that the minimizer is unique.

To the Monge problem of optimal transport is associated a related dual formulation of optimal transport, see Villani~\cite{Villani1}. The maximizers of such a problem include a function known as the potential function. The potential function is often intimately related to the optimal mapping. For instance, in the case of the squared distance cost $c(x,y) = \vert x - y \vert^2$, the optimal mapping (when it exists) is $T = \nabla \phi$, where the potential function $\phi$ is known as the Brenier potential after Yann Brenier who first rigorously proved the result in Brenier~\cite{brenier}. Thus, are therefore interested in the following quantities associated with the optimal transport problem, which are endlessly useful in applications and theory:
\begin{enumerate}
\item The total cost $C(\mu, \nu)$. For the choice $c(x,y) = \vert x - y \vert^{p}$, for $p \in [1, \infty)$ the $p$th root of this quantity is known as the Wasserstein-$p$ distance between $\mu$ and $\nu$, which is typically written as $W_{p}(\mu, \nu)$, that is:
\begin{equation*}
W_{p}(\mu, \nu) := \left( \int_{\mathbb{R}^d} \vert x - S(x) \vert^{p} d\mu(x) \right)^{1/p}.
\end{equation*}
\item The optimal mapping $T$ (may not necessarily exist and or be unique), that is $T$ such that $C(\mu, \nu) = \int_{\mathbb{R}^d} c(x,T(x)) d\mu(x)$ and $T_{\#} \mu = \nu$.
\item The potential function $\phi$ (is often unique up to a constant). In the case of $c(x,y) = \vert x - y \vert^2$, the potential function $\phi$ is known as the Brenier potential.
\end{enumerate}

\subsection{The Unbounded Support Problem and the Cutoff Approximation}

In this manuscript, we consider the Monge optimal transport problem in $n$-dimensional Euclidean space from a source measure $\mu$, supported on a compact set $\Omega \subset \mathbb{R}^n$ with density function $f$ to a target measure $\nu$, supported on $\mathbb{R}^n$ with density function $g$. The key feature of the problem we consider is the fact that the support of $\nu$ may be unbounded. Generally, we consider continuous, non-negative cost functions $c: \mathbb{R}^n \times \mathbb{R}^n \rightarrow \mathbb{R}$, but we will modify and/or restrict these assumptions later.

For various reasons, often motivated by applications, one wishes to approximate the target measure $\nu$ in some way. That is, define an arbitrary sequence $\{ \nu_{\rho} \}$ where $\nu_{\rho} \rightharpoonup \nu$ weakly as $\rho \rightarrow 0$, and study how the sequence of optimal transport problems changes as a function of $\rho$. That is, we may study how $C(\mu, \nu_{\rho})$, $T_{\rho}$ and $\phi_{\rho}$ converge to $C(\mu, \nu)$, $T$, and $\phi$, respectively (provided that $T$ is unique and $\phi$ has been uniquely identified). For example, it has been of longstanding interest to study such converge when $\nu_{\rho}$ is the empirical measure of $\nu$, see~\cite{berman},~\cite{stabilitywenbo}, among others.

In this manuscript, we will mostly consider the following way of approximating $\nu$ weakly. We define a ``cutoff" radius $R>0$ and define the cutoff probability measure $\nu_{R} := \nu/\nu(B_{R}(0))$, where $B_{R}(0)$ denotes the open Euclidean ball of radius $R$ centered at the origin. This is a natural way to approximate the optimal transport problem if one is motivated by using Monge-Amp\`{e}re equation solvers. We will call the optimal transport problem from $\mu$ to $\nu$ the forward problem and the optimal transport problem from $\nu$ to $\mu$ the inverse problem. We denote by $T_{\mu}$ and $\phi_{\mu}$ the optimal mapping and potential function, respectively, for the forward problem. Likewise, $T_{\nu}$ and $\phi_{\nu}$ denote the optimal mapping and potential function, respectively, for the inverse problem. For the cutoff problem, these will be denoted as $T_{\mu; R}$, $\phi_{\mu; R}$, $T_{\nu; R}$ and $\phi_{\nu; R}$, respectively. For the entire manuscript, we will be assuming that $\mu$ and $\nu$ have probability density functions. These will be denoted as $f_{\mu}$ and $f_{\nu}$, respectively. We denote the support of $\mu$ as the closed set $\Omega \subset \mathbb{R}^n$.

In Sections~\ref{sec:radial} and~\ref{sec:general}, we study the convergence of $T_{R}$ and $\phi_{R}$ to $T$ and $\phi$, respectively. We also consider the inverse optimal transport problem from $\nu_{R}$ to $\mu$. In Section~\ref{sec:radial} we derive convergence results in the case where $\mu$ and $\nu$ are radially symmetric for a wide variety of different cost functions. In Section~\ref{sec:general}, for the cost function $c(x,y) = \vert x - y \vert^2$, we derive quantitative convergence results, which we consider to be the main result of this manuscript. These results build off the $L^2$ convergence results of Delalande and M\`{e}rigot~\cite{merigot1}, for $\phi_{R}$ and almost-everywhere convergence results for $T_{R}$. We also derive a slew of results in anticipation that the $L^2$ convergence theory will be be expanded in the future.

\subsection{Discussion on Computational Methods and Applications}

Most existing algorithms available for computing optimal-transport-related quantities between a probability measure on $\Omega$ and a probability measure with unbounded support would involve sampling from the target measure. One approach would be to directly sample from both distributions a certain number of times and then compute the discrete Kantorovich plan using some modern variant of a network simplex algorithm, see Peyr\'{e} and Cuturi~\cite{Peyre_ComputationalOT}. In order to extract the optimal transport mapping one could then use a multigrid method as in Oberman and Ruan~\cite{obermanruan}. Another approach would be to sample from the target distribution and then use the tools of semi-discrete optimal transport, see~\cite{semidiscrete}, for example.

The second popular strategy is to use the entropically regularize the discrete optimal transport problem and then use the celebrated Sinkhorn algorithm, see Cuturi~\cite{cuturi} to get a ``smeared out" approximation of the Kantorovich plan. This computation is fast but difficult to use to discern the optimal transport mapping, although one can use the barycentric projection to approximate the optimal transport mapping, see Peyr\'{e} and Cuturi~\cite{Peyre_ComputationalOT}.

Another further interesting possibility would be to formulate the problem as an instance of a Moment Sum-of-Squares problem, see Lasserre~\cite{lasserre}, which can be used to estimate the support of the optimal transport mapping, see Mula and Nouy~\cite{momentsosOT}. This has not be done for unbounded problems, but could be applied to such problems, so it is not clear how well this works in practice.

The current manuscript does not address the issue where both $\mu$ and $\nu$ have unbounded support. However, the techniques in this manuscript may be used for certain such applications (where a pushforward map is needed) via ``bottlenecking". That is, we may compute a pushforward mapping $T$ from $\mu$ to $\nu$ via an intermediate step by mapping it to, say, a constant density on an n-cube of side length $2R$. In this way, we will not have an optimal transport map overall, but we do produce a pushforward map which can still be used for applications such as sampling and moving mesh methods. We cannot use the bottlenecking method if our goal is to compute the optimal transport map from $\mu$ to $\nu$.

For sampling methods, in Berman~\cite{berman} it was shown that if the Brenier potentials are $C^2$, then the rates are $\mathcal{O}(\sqrt{h})$ in the $H_1$ norm and if the Brenier potentials are not known to be smooth then this degrades to $\mathcal{O}(1/2^{d})$. Sampling methods are one of the only available tools in high dimensions, but still, the process of sampling, even if the density function is known via MCMC has some drawbacks and is slow when it needs to be used thousands or millions of times for many practical problems, see Moselhy and Marzouk for a more in-depth discussion~\cite{bayesianoptimalmaps}.

Our cutoff method is an attempt to derive bounds and better performance when (a) the density function is known and, especially, when (b) the integral of the density function over balls $B_{R}(0)$ can be efficiently computed. Decent convergence rates will be demonstrated under very mild assumptions. Excellent convergence rates will be demonstrated for a log-concave probability distributions, see Definition~\ref{def:logconcave}. This class of distributions includes sub-exponential (and therefore exponential), sub-Gaussian (and therefore Gaussian), as well as many more.

\subsection{Discussion on Quantitative $L^2$ Stability Bounds}

The $L^2$ bounds we utilize originate from a line of work involving many authors whose objective was to show whether or not the map $\mu \mapsto T_{\mu}$ is bi-H\"{o}lder, where $T_{\mu}$ is the optimal transport map from a measure $\mu$ to a measure $\nu$, both in Euclidean space, with finite second moments. A fundamental result from Andoni, Naor, and Neiman~\cite{snowflake} showed that the probability space $(\mathcal{P}_{2}(\mathbb{R}^n), W_{2})$ for dimension $n \geq 3$ does not admit a bi-H\"{o}lder embedding into a Hilbert space. That is, in general, we cannot hope for an inequality of the type:
\begin{equation}\label{eq:bounD}
\Vert T_{\mu_{1}} - T_{\mu_{2}} \Vert_{L^2(\mu)} \leq C W_{2}(\nu_{1}, \nu_{2}),
\end{equation}
in the most general case, where $T_{\mu_{1}}$ and $T_{\mu_{2}}$ are the optimal maps from $\mu$ to $\nu_1$ and $\nu_{2}$, respectively. An early result by Ambrosio and reported in Gigli~\cite{giglibounds} showed that a bound of the type in Inequality~\eqref{eq:bounD} could be obtained for $\mu$ and $\nu$ supported on a compact set, with $W_1$ instead of $W_2$ on the right-hand side, and assuming that the optimal mapping $T_{\mu_{1}}$ was $L$-Lipschitz. For our purposes, a natural predecessor to the Delalande and M\`{e}rigot result is due to Berman~\cite{berman}, where the author proved the following:
\begin{theorem}
Let $\mu$ have a probability density $f$ over a compact convex set $X$ and $f$ is bounded from above and below. Let $Y$ be a bounded connected open set with a Lipschitz boundary. Then, there exists a constant $C$ depending only on $\mu$, $X$ and $Y$ such that for any probability measures $\nu_{1}$ and $\nu_{2}$ on $Y$ we have:
\begin{equation}
\Vert T_{\nu_{1}} - T_{\nu_{2}} \Vert_{L^2(\mu)} \leq C W_{1}(\nu_{1}, \nu_{2})^{\alpha(n)},
\end{equation}
where $\alpha(n)$ is an exponent that depends only on the ambient dimension $n$.
\end{theorem}
This theorem is quite general, but has two drawbacks, which were later improved upon by Delalande and M\`{e}rigot~\cite{merigot1}. The first is that the exponent $\alpha(n)$ (details omitted) is suboptimal. The second drawback, most important for our purposes, is that the target domain $Y$ must be compact. It is precisely these two goals that motivated the work by Delalande and M\`{e}rigot. We will present their theorem in Theorem~\ref{thm:delalandemerigot}. Another benefit is that the work of Delalande and M\`{e}rigot supplied $L^2$ bounds on the convergence of the Brenier potentials. These, along with convexity properties of the Brenier potentials, allows for further convergence results, which we explore in this manuscript.

These estimates have garnered interest because they can be used to justify the validity of using linearized optimal transport for applications as well as justify numerical approximations of the problem, especially in the case where the measures (either the target measure or both) were approximated by empirical distributions.  On the the numerical side, there is the work by Li and Nochetto~\cite{stabilitywenbo}, in which a quantitative stability result was derived for approximating the transport plan from $\mu$ to $\nu$ on compact sets $X$ and $Y$, respectively, when both $\mu$ and $\nu$ are approximated by $\delta$ measures.

\subsection{Outline of Manuscript}

In Section~\ref{sec:definitions}, we introduce many definitions and fundamental theorems that will be used in this manuscript. In Section~\ref{sec:radial}, for a wide variety of cost functions, we show pointwise convergence rates for the cutoff problem and the inverse cutoff problem for the optimal mapping and the potential function in the case where the source and target probability measure is radially symmetric. We also demonstrate how quickly the rates of convergence can be for certain classes of distribution functions, including log-concave distribution functions. In Section~\ref{sec:general}, which constitutes the main results of this manuscript, we consider the cost function $c(x,y) = \frac{1}{2}c(x,y)^2$ and we relax the assumption of radial symmetry and use some fundamental $L^2$ bounds established by Delalande and M\'{e}rigot~\cite{merigot1} in order to derive many pointwise convergence rates for the Brenier potentials and inverse Brenier potentials for the cutoff problem. The establishment of these rates can be made quantitative in certain cases, but these pointwise convergence rates remain valid for any other approximation methods provided that rates for $L^1$ convergence have been established. We again show explicit rates in some cases like log-concave distribution functions. In the general case, we cannot derive pointwise convergence rates for the Monge mapping (or the inverse), but we derive almost-everywhere pointwise convergence guarantees. In Section~\ref{sec:computation} we give an example of a provably convergent monotone finite-difference approximation that can be used for the computation of the inverse Brenier potential, with the guidance of the theory established originally in Hamfeldt~\cite{HamfeldtBVP2}. In Section~\ref{sec:conclusion} we summarize the contributions of our manuscript.

\section{Definitions and Discussion}\label{sec:definitions}
In Section~\ref{sec:theoreticalbackground}, we introduce many of the useful concepts and results which will be used in this manuscript.

\subsection{Theoretical Background}\label{sec:theoreticalbackground}

Here we state the optimality conditions for optimal transport for the squared distance cost function, often known as Brenier's theorem. One can get a similar conclusion for more general cost functions satisfying the Ma-Trudinger-Wang conditions, see~\cite{MTW}, however the potential function is no longer convex, but $c$-convex, and the mapping is a more complicated function of the gradient of the $c$-convex potential.

\begin{theorem}[Optimality Conditions for Optimal Transport with Squared Distance Cost Function, see Loeper~\cite{LoeperReg}, for example]\label{thm:brenier}
Suppose $\mu, \nu \in \mathcal{P}_{2}(\mathbb{R}^d)$ and $\mu$ is absolutely continuous with respect to the Lebesgue measure on $\mathbb{R}^d$. Let $c(x,y) = \frac{1}{2} \vert x - y \vert^2$. Then, the optimal transport mapping $T$ exists and is unique almost everywhere in the support of $\mu$. It satisfies the equation $T_{\#}\mu = \nu$. Furthermore, there exists a lower semi-continuous convex function $u$, known as the Brenier potential function, such that $T(x) = \nabla u(x)$ almost everywhere in the support of $\mu$. Note that $v = u+C$ also satisfies $T(x) = \nabla v(x)$ almost everywhere in the support of $\mu$, so $u$ is unique up to a constant.

The converse also holds. Suppose that you have a lower semi-continuous convex function $u$. Define $T = \nabla u(x)$ almost everywhere. If $T$ also satisfies $T_{\#}\mu = \nu$, then $T$ is an optimal mapping. Optimal maps are equal $\mu$-a.e. and $\text{Cl}(\nabla u(F \cap \text{spt}(\mu))) = \text{spt}(\nu)$ where $F$ is the set of differentiability points of $u$.

\end{theorem}
\begin{remark}
For the purposes of this paper, if $\mu$ is supported on a compact set $\mathcal{X}$, then we only consider potential functions satisfying $\int_{\mathcal{X}}u(x)dx = 0$, which will uniquely define this function on $\mathcal{X}$. Notice that the mapping $T$ is not uniquely defined everywhere. In this paper, for the radial case, we will define particular $T$ and Brenier potential $u$ using cumulative distribution functions that avoids this technical issue. For the general case, we will be able to prove explicit convergence rates for sequences of Brenier potential functions $u$ in an open set, but only almost everywhere convergence for sequences of mappings $T$. Thus, due to the possible discontinuities of $T$ and due to the fact that the optimal maps are only defined $\mu$-a.e., we cannot expect to do better than almost everywhere convergence.
\end{remark}

In convex analysis, it is useful to define a convex function in $\mathbb{R}^d$ on the entire Euclidean space and also allow it to be defined on the extended real numbers, i.e. it can attain the values $\pm \infty$. Then, a convex function $u$ can be defined to be a function whose secant lines lie above the graph of the function, i.e. $tu(x) + (1-t)u(y) \geq u(tx + (1-t)y)$ for any $x, y \in \mathbb{R}^d$ and any $t \in [0,1]$. While this definition is simple and extraordinarily useful, it does lead to some complications. A convex function could be improper, which means that it is either $+\infty$ everywhere or equal to $-\infty$ somewhere. Thus, one usually defines the essential domain $\text{dom}(u)$ to be the set where $u$ is finite. On the boundary of the essential domain, $u$ could be discontinuous. However, if $\text{dom}(u)$ has a non-empty interior, then $u$ is reasonably well-behaved in the interior of the domain. It is continuous in the interior of the domain and is Lipschitz on any closed subset of the interior. However, since $u$ is not guaranteed to be differentiable, but the graph of $u$ admits supporting hyperplanes, it is possible to define a subdifferential at each point in $\text{dom}(u)$. We thus introduce the subdifferential:

\begin{definition}\label{def:subdifferential}
The subdifferential of a convex function $u$ at a point $x$, denoted by $\partial u(x)$ is defined as the following set-valued function:
\begin{equation}
\partial u(x) := \{y \in \mathbb{R}^d : f(z) \geq f(x) + y \cdot (z - x) \ \forall z \in \mathbb{R}^d\}.
\end{equation}
\end{definition}
For a set $\Omega \in \mathbb{R}^d$, we define $\partial u(\Omega) := \cup_{x \in \Omega} \partial u(x)$. With these definitions, we can state the following results:
\begin{theorem}
Let $u$ be a convex function. Let $\text{dom}(u)$ denote the set where $u$ is finite and let $F$ denote the set of differentiability of $u$. Then, $\partial u(x) \neq \emptyset$ for every $x \in \text{dom}(u)$. Furthermore, if $u$ is differentiable at a point $x$, then $\partial u(x) = \nabla u(x)$. Finally, the subdifferential has the following continuity property. For each point $x \in F$ and for every $\epsilon>0$, there exists a $\delta>0$ such that $\partial u (B_{\delta}(x)) = B_{\epsilon}(\nabla u(x))$.
\end{theorem}

Once enough regularity is assumed, for the squared distance cost function $c(x,y) = \vert x - y \vert^2$ the optimal transport potential function $u$ will solve the following fully nonlinear partial differential equation which is elliptic on the space of convex functions:
\begin{definition}[Monge-Amp\`{e}re equation with second boundary value condition]
The partial differential equation
\begin{equation}
\det(D^2 u(x)) = f(x)/g(\nabla u(x)),
\end{equation}
is known as an instance of the Monge-Amp\`{e}re equation. Let $\Omega = \text{spt}(\mu)$ and $\Omega ' = \text{spt}(\nu)$. Subject to the global condition: $\partial u(\Omega) \subset \Omega '$, then there is a unique convex $u$ up to an additive constant, see Urbas~\cite{Urbas}.
\end{definition}

In $\mathbb{R}$, under mild conditions the optimal transport mapping has an explicit formula for a wide variety of cost functions. The formula comes from the observation that a monotone mapping $T(x)$ is optimal. If the mapping $T$ is monotone, then it can only be of a certain form in $\mathbb{R}$. Here, we treat the case where $\mu$ is absolutely continuous with respect to the Lebesgue measure on $\mathbb{R}$. The result is adapted from the more general statement in Maggi~\cite{maggi}.

\begin{theorem}[Monge problem of optimal transport in $\mathbb{R}$]\label{thm:onedim}
If $\mu$ is absolutely continuous with respect to the Lebesgue measure in $\mathbb{R}$ and $\mu, \nu \in \mathcal{P}_{1}(\mathbb{R})$ and $c(x,y) = h(\vert y - x \vert)$, where $h(0) = 0$, $h$ is increasing and strictly convex, then letting $F, G$ be the cumulative distribution functions of $\mu$ and $\nu$ and defining $G^{[-1]}:= \inf \{ x \in \mathbb{R}: G(x) \geq t \}$, then the optimal transport mapping $T$ is given by the following formula: $T(x) = G^{[-1]} \circ F(x)$. In the case that $\nu$ does not have any ``gaps", then $G$ is invertible and we have $T(x) = G^{-1} \circ F(x)$.
\end{theorem}
\begin{remark}
If the cost function does not satisfy these conditions, the formula may not hold. The simplest example where this can happen is $c(x,y) = \vert y - x \vert$, and $\mu = \mathbf{1}_{[0,2]}dx$ and $\nu = \mathbf{1}_{[1,3]}dx$. In this case, there is a monotone rearrangement $T(x) = x + 1$ that is optimal, but it is not the only rearrangement that is optimal. For example, the function
\begin{equation}
S(x) = \begin{cases}
x, &x \in [1,2], \\
x + 2, &x \in [0, 1],
\end{cases}
\end{equation}
is also an optimal transport mapping, which is not monotone.
\end{remark}

Finally, in this manuscript, we will consider probability distributions whose $p$th moment is finite for some $p \in \mathbb{N}^{+}$. However, the strongest results will be shown for a class of probability distributions that are known as log-concave. Not only do their density function decay ``quickly", but they also have nice properties like the continuity and monotonicity of their density functions.

\begin{definition}\label{def:logconcave}
We say that a probability distribution $\mu$ on $\mathbb{R}^n$ is log-concave if there exists a Lebesgue integrable distribution function $f: \mathbb{R}^n \rightarrow \mathbb{R}$ that satisfies $\mu(E) = \int_{E} f(x)dx$ for any Borel set $E \subset \mathbb{R}^n$ and $f(x) = e^{-\varphi(x)}$, where $\varphi(x)$ is a convex function.
\end{definition}

\section{Rapid Convergence for Radial Case}\label{sec:radial}
In this section, we state convergence rates that apply for a wide variety of cost functions for source and target measures that are radially symmetric. We assume that $\mu$ and $\nu$ have density functions $f_{\mu}$ and $f_{\nu}$, respectively.

Assuming some conditions on the source and target measures, the rates of convergence of the optimal transport map, inverse optimal transport map, potential function and inverse potential function are controlled by the quantity $1 - F_{\nu}(R)$, where,
\begin{equation}\label{eq:convergencerate}
1 - F_{\nu}(R) = \vert \mathbb{S}^{n-1} \vert \int_{R}^{\infty} \tilde{f}_{\nu}(\tau) \tau^{n-1} d\tau,
\end{equation}
where $\tilde{f}_{\nu}(\vert x \vert) = f_{\nu}(x)$. The quantity in Equation~\eqref{eq:convergencerate} is a sort of radial cumulative distribution function. As will be shown, $1 - F_{\nu}(R) = \mathcal{O}(R^{-p})$ for target distributions $\nu$ with finite $p$th moment and $1 - F_{\nu}(R) = \mathcal{O}(R^{n-1} e^{-y R})$, for some $y>0$ for $\nu$ that are log-concave.

In Section~\ref{sec:radialsolutions}, we prove that radial measures will lead to radial potential functions and radial optimal mappings. In Section~\ref{sec:radialmap} we prove the convergence rate for inverse radial maps. In Section~\ref{sec:inverseradialmap}, we show under what conditions a similar rate of convergence can be established for the forward radial maps. In Section~\ref{sec:radialBrenier}, we establish the rate of convergence for the inverse potential functions and in Section~\ref{sec:inverseradialBrenier} we show how a similar bound can be achieved for the forward radial potential functions. Finally, in Section~\ref{sec:radialdistributions}, we show how we can establish the expected rates for radial distributions with 1) finite $p$th moments and 2) when the radial distribution is log-concave.

\subsection{Radial Source and Target Distributions Lead to Radial Solutions}\label{sec:radialsolutions}

Our goal is to establish explicit $L^{\infty}$ convergence rates for both the optimal transport map and the Brenier potential for the forward and inverse problems. We begin with a theorem that shows that when $\nu$ is radially symmetric and for certain cost functions, the optimal transport mapping and the potential functions are radially symmetric and have explicit formulas which are derived from the formula for one-dimensional optimal transport, see Section 16 from Maggi~\cite{maggi} for an excellent exposition on one-dimensional optimal transport.

\begin{theorem}\label{thm:radial}
If the cost function $c(x,y) = h(\left\vert x - y \right\vert)$, where $h(z)$ is increasing and strictly convex with $h(0) = 0$ and if $f_{\mu}$ and $f_{\nu}$ are radially symmetric, i.e. $f(Ox) = f_{\mu}(x)$ and $f_{\nu}(Ox,) = f_{\nu}(x)$ for any orthogonal matrix $O \in \mathbb{R}^{n^2}$, then the optimal transport map $T_{\nu}$ and the potential function $\phi_{\nu}$ are both radially symmetric. Since $f_{\mu}$ and $f_{\nu}$ are radially symmetric, there exist functions $\tilde{f}_{\mu} : [0, \infty) \rightarrow [0, \infty)$ and $\tilde{f}_{\nu}:[0, \infty) \rightarrow [0, \infty)$ such that $\tilde{f}_{\mu}(\vert x \vert) = f_{\mu}(x)$ and $\tilde{f}_{\nu}(\vert x \vert) = f_{\nu}(x)$. The optimal mapping is then given explicitly by the formula
\begin{equation}\label{eq:radialmapping}
T_{\nu}(x) = \frac{x}{\vert x \vert}F_{\mu}^{-1} \circ F_{\nu}(\vert x \vert)
\end{equation}
for $x \neq 0$ and $T(0) = 0$, where $F_{\mu}$ and $F_{\nu}$ are defined as:
\begin{align}
F_{\mu}(t) &:= \left\vert \mathbb{S}^{n-1} \right\vert \int_{0}^{t} \tilde{f}_{\mu}(\tau) \tau^{n-1}d\tau, \\
F_{\nu}(t) &:= \left\vert \mathbb{S}^{n-1} \right\vert \int_{0}^{t} \tilde{f}_{\nu}(\tau) \tau^{n-1}d\tau.
\end{align}
Furthermore, every corresponding potential function $\phi_{\nu}$ is $c$-concave and radially symmetric and is given by the formula:
\begin{equation}\label{eq:radialpot}
\phi_{\nu}(x) = h \left( \left\vert x  - F_{\mu}^{-1} \circ F_{\nu} (\vert x \vert) \right\vert \right) + C,
\end{equation}
for some constant $C \in \mathbb{R}$.
\end{theorem}

\begin{remark}
This result applies to both $\nu$ with unbounded support and also the cutoff measure $\nu_{R}$.
\end{remark}

\begin{proof}

First, we note that If $\tilde{f}_{\nu}(t) \neq 0$, then the map $T_{\nu}$ is invertible. In order to show that Equation~\eqref{eq:radialmapping} is the optimal map. First, we show that $T_{\nu}$ satisfies $(T_{\nu})_{\#}\nu = \mu$. We then demonstrate that any $\phi_{\nu}$ satisfying $\frac{d}{dx} \phi_{\nu}(x)= -\frac{\partial }{ \partial x} c(x,T_{\nu}(x))$ is a $c$-concave function. By the necessary and sufficient optimality conditions for optimal transport mapping~\cite{MTW}~\cite{LoeperReg}, we then know that $T_{\nu}$ is optimal for the cost function $c:\mathbb{R}^n \times \mathbb{R}^n \rightarrow \mathbb{R}$.

First, we show that for any Borel set $E \subset \mathbb{R}^{n}$, we have $\nu(T_{\nu}^{-1}(E)) = \mu(E)$, where $T_{\nu}^{-1}(E) := \{ x \in \mathbb{R}^{n-1}: T_{\nu}(x) \in E \}$. In order to do this, we demonstrate that this holds for any sector. Let $(r, \theta_1, \dots, \theta_{n-1})$ be the standard spherical coordinate system in $\mathbb{R}^n$. Then, we define a sector $S \subset \mathbb{R}^n$ to be the following type of closed set:
\begin{equation}\label{eq:radialchunk}
S = \{x \in \mathbb{R}^n : r_1 \leq r \leq r_2, \theta_{1, 1} \leq \theta_{1} \leq \theta_{1, 2} , \dots, \theta_{n-1, 1} \leq \theta_{n-1} \leq \theta_{n-1, 2} \}
\end{equation}
for given $\theta_{i, j}$ where $i = 1, \dots, n-1$ and $j = 1, 2$ and $r_1, r_2$, where $r_1 \leq r_2$ and $\theta_{i, 1} \leq \theta_{i, 2}$ for all $i$. We also define the set $\mathcal{S} :=\{ A \subset \mathbb{R}^n : A \ \text{is a sector} \}$. Showing that $T_{\nu}$ is a pushforward map for such sets shows the result in general.

First, $T_{\nu} = F_{\mu}^{-1} \circ F_{\nu}$ is monotone and therefore Borel measurable. Therefore, for any Borel set $B \subset \mathbb{R}^n$, we have that $T_{\nu}^{-1}(B)$ is a Borel set.

Second, any open set $A \subset \mathbb{R}^n$ can be written as a countable union of sectors. We demonstrate this fact for bounded open Borel sets. The unbounded case follows by taking countable unions of bounded open sets. Let $B$ be unbounded. Then, there exists a countable sequence of disjoint balls $T_i := \{ x \in \mathbb{R}^n : \vert x \vert < 2^{i} \}$ for $i=0, 1, 2, \dots$ such that $B = \cup_{i=1}^{\infty} C_i$, where $C_i = B \cap T_{i}$ are bounded open sets.

We now demonstrate that given a bounded open Borel set $B \subset \mathbb{R}^n$, we can construct a countable number of sectors $S_i$ such that $B = \cup_{i} S_i$. Fix a sector $S$ such that $S \subset A$. This is possible since $A$ is open. Let the Lebesgue measure of $S$ be defined as $R$, i.e. $\vert S \vert =:R>0$. Let $S_{1, 1}$ be a sector of Lebesgue measure $\vert S_{1,1} \vert \in [R, R/2]$ such that $S_{1,1} \subset A$. Then, form $A_{1,1} = A \setminus S_{1,1}$. Then $A_{1, 1}$ is open. Choose $S_{1, 2} \in A_{1, 1}$ such that $\vert S_{1, 2} \vert \in [R, R/2]$ and form $A_{1, 2}$. Keep doing this process until the first integer $k$ such that there is no sector $S$ of Lebesgue measure in $[R,R/2]$ such that $S \in A_{1, k-1}$. This defines a set of $k-1$ sectors, $\{S_{1, i} \}_{i=1}^{k-1}$ and an open set $A_{1, k-1}$. Let $j\geq 1$ be the smallest integer such that there exists a sector $S$ that satisfies $\vert S \vert \in [R/2^{j}, R/2^{j+1}]$ and $S \subset A_{1, k-1}$. Such an integer exists, since $A_{1, k-1}$ is open and is non-empty. Then, let $S_{j, 1}$ be a sector such that $\vert S_{j, 1} \vert \in [R/2^{j}, R/2^{j+1}]$ and $S_{j, 1} \subset A_{1, k-1}$. Define $A_{j, 1} = A_{1, k-1} \setminus S_{j, 1}$. Repeat this process again until there exists the first integer $k'$ such that there is no sector $S$ satisfying $\vert S \vert \in [R/2^{j}, R/2^{j+1}]$ and $S \subset A_{j, k'-1}$. Now, continue shrinking the interval until there is an integer $j'$ such that there exists a sector $S \subset A_{j, k'+1}$ such that $\vert S \vert \in [R/2^{j'}, R/2^{j'+1}]$. We continue this procedure removing volume from $A$ until the Lebesgue measure of the sectors goes to zero. If there were volume left after this process, we would have a non-empty bounded open set $A'$, and by definition, there would exists a constant $R'>0$ such that a sector $S$ of Lebesgue measure $R'$ would satisfy $S \subset A'$, contradicting that there is no volume left after the procedure. Hence, the set $A'$ must have zero Lebesgue measure and be open. The only such set is the empty set. Therefore, there exists a countable collection $\{ S_l\}$ of sectors $S_l$ such that $A = \cup_{l=1}^{\infty} S_{l}$.

We define the set:
\begin{equation}
\mathcal{L} := \{ B \subset \mathbb{R}^n : B \ \text{Borel}, \nu(T_{\nu}^{-1}(B)) = \mu(B)\}.
\end{equation}
Thus, this set is a subset of the Borel $\sigma$-algebra. Since $\sigma(\mathcal{S})$ generates the open sets, then the Borel $\sigma$-algebra is contained in $\sigma(\mathcal{S}))$. Therefore, we have:
\begin{equation}
\mathcal{L} \subset \mathcal{B} \subset \sigma(\mathcal{S}).
\end{equation}
We now show that $\sigma(\mathcal{S}) \subset \mathcal{L}$. This will follow from Theorem 1.4 ($\pi$-$\lambda$ Theorem) in Evans and Gariepy~\cite{evansgariepy} once we have established that $\mathcal{L}$ is a $\lambda$-system. The fact that $\mathcal{S}$ is a $\pi$-system follows from the definition of the sectors. If $\sigma(\mathcal{S}) \subset \mathcal{L}$, then $\mathcal{L} = \sigma(\mathcal{S}) = \mathcal{B}$. Therefore, $\nu(T_{\nu}^{-1}(B)) = \mu(B)$ for all Borel sets $B$. This shows that it is sufficient to check that $\nu(T_{\nu}^{-1}(S)) = \mu(S)$ for all sectors $S$. Importantly, $\mathcal{S} \subset \mathcal{L}$, since the sectors are examples of Borel sets $B$ for which we have verified that $\nu(T_{\nu}^{-1}(B)) = \mu(B)$.

Claim: $\mathcal{L}$ is a $\lambda$-system. Firstly, since $T_{\nu}:\mathbb{R}^n \rightarrow \mathbb{R}^n$, we have $\mathbb{R}^n \in \mathcal{L}$. Secondly, let $A, B \in \mathcal{L}$ with $B \subset A$. Then, $\mu(A \setminus B) = \mu(A) - \mu(B) = \nu(T_{\nu}^{-1}(A)) - \nu(T_{\nu}^{-1}(B)) = \nu(T_{\nu}^{-1}(A) \setminus T_{\nu}^{-1}(B))$, since $A, B, T_{\nu}^{-1}(A), T_{\nu}^{-1}(B)$ are measurable. Also, $\nu(T_{\nu}^{-1}(A) \setminus T_{\nu}^{-1}(B))= \nu(T_{\nu}^{-1}(A \setminus B))$, since $x \in T_{\nu}^{-1}(A \setminus B)$ means $T_{\nu}(x) \in A \setminus B$ which means $T_{\nu}(x) \in A$ and $T_{\nu}(x) \notin B$, which implies that $x \in T_{\nu}^{-1}(A)$ and $x \notin T_{\nu}^{-1}(B)$ which implies $x \in T_{\nu}^{-1}(A) \setminus T_{\nu}^{-1}(B)$. The same argument in reverse shows that $x \in T_{\nu}^{-1}(A) \setminus T_{\nu}^{-1}(B)$ implies $x \in T_{\nu}^{-1}(A \setminus B)$. Therefore, $\nu(T_{\nu}^{-1}(A \setminus B)) = \mu(A \setminus B)$.

Now, let $A_k \in \mathcal{L}$ for each $k$ and let $A_{k} \subset A_{k+1}$. Then, by Theorem 1.6 in Evans and Gariepy~\cite{evansgariepy}, we have that since $\mu$ is a Radon measure $\mu(\cup_{k=1}^{\infty} A_k) = \lim_{k \rightarrow \infty} \mu(A_k) = \lim_{k \rightarrow \infty} \nu(T_{\nu}^{-1}(A_k)) = \nu (\cup_{k=1}^{\infty} T_{\nu}^{-1}(A_k)$, where the second equality follows since $A_k \in \mathcal{L}$. We now show that $T_{\nu}^{-1}(\cup_{k=1}^{\infty} A_k) = \cup_{k=1}^{\infty} T_{\nu}^{-1}(A_k)$ from which the result follows. Let $x \in T_{\nu}^{-1}(\cup_{k=1}^{\infty} A_k)$. Then, $T_{\nu}(x) \in \cup_{k=1}^{\infty} A_k$. Then, since $A_k$ are increasing sets, this then says that there exists a $K$ such that for all $k \geq K$, $T_{\nu}(x) \in A_k$. Let $y \in \cup_{k=1}^{\infty} T_{\nu}^{-1}(A_k)$. Then, there exists a $K'$ such that for all $k \geq K'$ we have $y \in T_{\nu}^{-1}(A_k)$. Thus, $T_{\nu}(y) \in A_k$ for $k \geq K'$. Thus, $y \in T_{\nu}^{-1}(\cup_{k=1}^{\infty} A_k)$ and $x \in \cup_{k=1}^{\infty} T_{\nu}^{-1}(A_k)$. Therefore, they are equal. We have thus shown that $\mathcal{L}$ is a $\lambda$-system and therefore the claim follows.


We define the quantity $V = \int_{\theta_{1, 1}}^{\theta_{1, 2}} \dots \int_{\theta_{n-1, 1}}^{\theta_{n-1, 2}} d\theta_{1} \dots d\theta_{n-1}$. If $f_{\nu}(t) \neq 0$, then $T_{\nu}$ is invertible and is given by
\begin{equation}
T_{\mu}(y) = F_{\nu}^{-1} \circ F_{\mu}(\vert y \vert) \hat{y}.
\end{equation}

More generally, if $f_{\nu}$ is allowed to equal zero, then we can define a generalized inverse of $T_{\nu}$:
\begin{equation}
T_{\mu}(y) = F_{\nu}^{[-1]} \circ F_{\mu}(\vert y \vert) \hat{y},
\end{equation}
where $F_{\nu}^{[-1]}(t) := \inf \{ x \in \mathbb{R} : F_{\nu}(x) \geq t \}$. Note that since $F_{\nu}^{[-1]}(t)$ is still monotone, so if $B_1 \cap B_2 = \emptyset$, then the preimage of $B_1$ via $T$ is disjoint from the preimage of $B_2$ via $T_{\nu}$. We compute:
\begin{align}
\mu(B) &= \int_{B} f_{\mu}(x) dx, \\
&= \int_{\theta_{1, 1}}^{\theta_{1, 2}} \dots \int_{\theta_{n-1, 1}}^{\theta_{n-1, 2}} \int_{r_{1}}^{r_{2}} r^{n-1} \tilde{f}_{\mu}(r) dr d\theta_{1} \dots d\theta_{n-1}, \\
&= V \int_{r_{1}}^{r_{2}} r^{n-1} \tilde{f}_{\mu}(r) dr, \\
&= \frac{V}{\vert \mathbb{S}^{n-1} \vert}(F_{\mu}(r_2) - F_{\mu}(r_1)).
\end{align}
We also compute:
\begin{align}
\nu (T_{\nu}^{[-1]}(B)) &= \int_{T_{\nu}^{[-1]}(B)} f_{\nu}(x) dx, \\
&= \int_{\theta_{1, 1}}^{\theta_{1, 2}} \dots \int_{\theta_{n-1, 1}}^{\theta_{n-1, 2}} \int_{F_{\nu}^{[-1]} \circ F_{\mu}(r_{1})}^{F_{\nu}^{[-1]} \circ F_{\mu}(r_{2})} r^{n-1} \tilde{f}_{\nu}(r) dr d\theta_{1} \dots d\theta_{n-1}, \\
&= V \int_{F_{\nu}^{[-1]} \circ F_{\mu}(r_{1})}^{F_{\nu}^{[-1]} \circ F_{\mu}(r_{2})} r^{n-1} \tilde{f}_{\nu}(r) dr, \\
&= \frac{V}{\vert \mathbb{S}^{n-1} \vert} (F_{\nu} \circ F_{\nu}^{[-1]} \circ F_{\mu}(r_2) - F_{\nu} \circ F_{\nu}^{[-1]} \circ F_{\mu}(r_1)), \\
&= \frac{V}{\vert \mathbb{S}^{n-1} \vert} (F_{\mu}(r_2) - F_{\mu}(r_1)), \\
&= \mu(B).
\end{align}

Thus, we have shown that $(T_{\nu})_{\#}\nu = \mu$.


Define $\hat{f}_{\nu}(t):= \vert \mathbb{S}^{n-1} \vert \tilde{f}_{\nu}(t) t^{n-1}$ and $\hat{f}_{\mu}(t):= \vert \mathbb{S}^{n-1} \vert \tilde{f}_{\mu}(t) t^{n-1}$. Since these are both densities on $[0, \infty)$, we define an optimal transport problem between $\hat{f}_{\nu}$ and $\hat{f}_{\mu}$ on $[0, \infty)$. The resulting optimal map, via Theorem~\ref{thm:onedim}, is given by $S(r) = F_{\mu}^{-1} \circ F_{\nu}(r)$. We then note that $T_{\nu}(x) = S(\vert x \vert)\hat{x}$.

In order to show that $\phi_{\nu}$ is $c$-concave, we need to show that there exists a function $\beta: \mathbb{R}^n \rightarrow \mathbb{R} \cup \{-\infty \}$ such that:
\begin{equation}
\phi_{\nu}(x) = \inf_{y \in \mathbb{R}^{n}} \{ h(\vert x - y \vert) - \beta(y) \}.
\end{equation}
We define $w:\mathbb{R} \rightarrow \mathbb{R}$ via
\begin{equation}
w(r) = h(\vert r - S(r) \vert),
\end{equation}
we know that $w$ is $c$-concave for the optimal transport problem on $[0, \infty)$, since $S$ is optimal. We thus know that there exists a function $\alpha: \mathbb{R} \rightarrow \mathbb{R}$ such that:
\begin{equation}
w(r) = \inf_{r' \in \mathbb{R}} \{ h(\vert r - r' \vert) - \alpha(r') \}.
\end{equation}
Since for $x'$ that satisfies $x = a x'$ for some $a \geq 0$, we have:
\begin{align}
\inf_{y \in \mathbb{R}^{n}} \{ h(\vert x - y \vert)-\alpha(y)  \} &\leq \inf_{y \in \mathbb{R}^{n}, y = ax, a\geq 0} \{ h(\vert x - y \vert) - \alpha(y) \} \\
&= \inf_{y \in \mathbb{R}^{n}, y = ax, a\geq 0} \{ \alpha(y) - h(\vert \vert x \vert - \vert y \vert \vert) \} = w(\vert x \vert).
\end{align}
But, by the reverse triangle inequality: $\vert \vert x \vert - \vert y \vert \vert \leq \vert x - y \vert$. Since $h$ is an increasing function of its argument, we have:
\begin{equation}
h(\vert \vert x \vert - \vert y \vert \vert) \leq h(\vert x - y \vert).
\end{equation}
So, then we have:
\begin{align}
\inf_{y \in \mathbb{R}^{n}} \{ h(\vert x - y \vert)- \alpha(y) \} &\geq \inf_{y \in \mathbb{R}^{n}} \{ h(\vert  \vert x \vert - \vert y \vert \vert) -\alpha(y)\} \\
&= w(\vert x \vert).
\end{align}
Hence, we have found that there exists a $\beta(y)$ satisfying $\beta(y) = \alpha(\vert y \vert)$ such that:
\begin{equation}
\phi_{\nu}(x) = w(\vert x \vert) = \sup_{y \in \mathbb{R}^n} \{ \beta(y) - h( \vert x -y \vert) \},
\end{equation}
and hence $\phi_{\nu}$ is $c$-convex and thus $T_{\nu}$ is the optimal transport map from $\nu$ to $\mu$. We may choose the constant $C$ in Equation~\eqref{eq:radialpot} to be equal to zero.

\end{proof}

\begin{remark}
This theorem can, of course, be generalized, see Thorpe~\cite{thorpe} and Maggi~\cite{maggi}. For example, strict convexity of $h$ can also be relaxed to merely convexity. Also, the condition that $h(0) = 0$ is not vital, as the cost function can be redefined by a constant without changing the optimizers of the optimal transport problem.
\end{remark}

Here we prove a lemma which will be useful to us in the sequel:
\begin{lemma}\label{lemma:lipschitzinverse}
Let $\mu$ be a radially symmetric measure on $\mathbb{R}^n$ with integrable density function $f$. Define $\tilde{f}(\vert x \vert) = f(x)$ and $F(r) := \vert \mathbb{S}^{n-1} \vert \int_{0}^{r} \tilde{f}(r) r^{n-1}dr$. Let $\tilde{f}$ satisfy $\tilde{f}(t) \geq m(\tau)>0$ for $t \in [0, \tau]$. Then, $F$ is strictly monotonic and $F^{-1} : [0, 1] \rightarrow [0, \infty)$ exists and $F^{-1}$ satisfies the following bounds:
\begin{equation}
\vert F^{-1}(t_2) - F^{-1}(t_1) \vert \leq \frac{n}{m(\tau) \vert \mathbb{S}^{n-1} \vert \min \{ r_2^{n-1}, r_1^{n-1} \} } \vert r_2 - r_1 \vert,
\end{equation}
for $t_1, t_2$ such that $0 \leq F^{-1}(t_1), F^{-1}(t_2) \leq \tau$.
\end{lemma}
\begin{proof}
Without loss of generality, suppose that $r_2 \geq r_1$. Since $F$ is strictly monotone it has a well-defined inverse $F^{-1}$. Then,
\begin{align}
\vert F^{-1}(t_2) - F^{-1}(t_1)  \vert &= \vert r_2 - r_1 \vert \\
&= \frac{n}{\sum_{i=0}^{n-1} r_1^i r_2^{n-1-i}} \int_{r_1}^{r_2} r^{n-1}dr \\
&\leq \frac{n}{m(\tau) \sum_{i=0}^{n-1} r_1^i r_2^{n-1-i}} \int_{r_1}^{r_2} \tilde{f}(r) r^{n-1}dr \\
&= \frac{n}{m(\tau) \vert \mathbb{S}^{n-1} \vert \sum_{i=0}^{n-1} r_1^i r_2^{n-1-i}} \vert t_2 - t_1 \vert \\
&\leq \frac{1}{m(\tau) \vert \mathbb{S}^{n-1} \vert  \min \{r_1^{n-1}, r_2^{n-1} \} } \vert t_2 - t_1 \vert.
\end{align}
\end{proof}
\begin{remark}
We will apply the result of this Lemma to both the source and target distributions. Since the target distribution has compact support, the choice of the constant $m(\tau) = m$ will be uniform. For the source distribution, the constant $m(\tau)$ will depend on $\tau$, as suggested by the notation.
\end{remark}

\subsection{Convergence Rates of Inverse Optimal Radial Mapping $T_{\nu}$}\label{sec:radialmap}

Using the explicit form of the optimal transport mapping given in Theorem~\ref{thm:radial}, we can compute convergence rates of the map $T_{\nu; R}$ to $T_{\nu}$. Under reasonable circumstances, the rate will be very fast, that is $\mathcal{O}(1 - F_{\nu}(R))$.

\begin{proposition}\label{prop:radialmap}
Let the cost function $c(x,y) = h(\left\Vert x - y \right\Vert)$, where $h(z)$ is increasing and strictly convex with $h(0) = 0$. Let there exist $m>0$ such that $\tilde{f}_{\mu} \geq m>0$ for $x \in \Omega$ and let $\tilde{f}_{\nu} >0$. Then, we have, for $R \geq R_0$, where $R_0$ is the smallest value for which $F_{\nu}(R_0) \geq 1/2$ the following estimate:
\begin{equation}\label{eq:rate}
\vert T_{\nu;R}(x) - T_{\nu}(x) \vert \leq C(m, n, \vert x \vert) (1-F_{\nu}(R)),
\end{equation}
where the constant is given in Inequality~\eqref{eq:radialmap}.
\end{proposition}

\begin{proof}
For $x = 0$, we have $T_{\nu}(x) = T_{\nu; R}(x) = 0$. For $x \neq 0$ that:
\begin{equation}
\Vert T_{\nu}(x) - T_{\nu; R}(x) \Vert = \left\vert F_{\mu}^{-1} \circ F_{\nu}(\vert x \vert) - F_{\mu}^{-1} \circ \left( \frac{F_{\nu}(\vert x \vert )}{F_{\nu}(R)} \right) \right\vert.
\end{equation}
By Lemma~\ref{lemma:lipschitzinverse} applied to $F_{\mu}$, we get:
\begin{align}
\left\vert F_{\mu}^{-1} \circ F_{\nu}(\vert x \vert) - F_{\mu}^{-1} \circ \left( \frac{F_{\nu}( \vert x \vert)}{F_{\nu}(R)} \right) \right\vert &\leq \frac{1}{m \vert \mathbb{S}^{n-1} \vert (F_{\mu}^{-1} \circ F_{\nu}(\vert x \vert))^{n-1}} \left( F_{\nu}(\vert x \vert) - \frac{F_{\nu}(\vert x \vert)}{F_{\nu}(R)} \right)\\
&= \frac{F_{\nu}(\vert x \vert) }{m \vert \mathbb{S}^{n-1} \vert (F_{\mu}^{-1} \circ F_{\nu}(\vert x \vert))^{n-1} F_{\nu}(R)} \left( 1 - F_{\nu}(R) \right).
\end{align}

We now use the fact that $F_{\nu}(\vert x \vert) \leq 1$ and we choose $R\geq R_0$ where $R_0$ is the smallest value for which $F_{\nu}(R_0) \geq \frac{1}{2}$. Furthermore, since $f_{\nu}(x)>0$, if $\vert x \vert>0$, then $F_{\nu}(\vert x \vert) = \int_{0}^{\vert x \vert} \tilde{f}_{\nu}(r) r^{n-1}dr>0$. Thus, there exists a monotonically increasing function $p:[0, \infty) \rightarrow [0, \frac{1}{2} \text{diam}(\Omega)]$ with $p(t)>0$ for $t>0$ such that the value of $r$ such that $F_{\nu}(\vert x \vert) = \int_{0}^{r} \tilde{f}_{\mu}(t) t^{n-1} dt$ satisfies $r= p(\vert x\vert)$ where $p(t)>0$ for $t>0$. Thus, for $R \geq R_0$, we have:
\begin{align}\label{eq:radialmap}
\left\vert T_{\nu;R}(x) - T_{\nu}(x) \right\vert &\leq \frac{2}{m \vert \mathbb{S}^{n-1} \vert p(\vert x \vert)^{n-1}}(1-F_{\nu}(R)) \\
&= C(m, n, \vert x \vert) (1-F_{\nu}(R)).
\end{align}
\end{proof}


\subsection{Convergence of Radial Map $T_{\mu}$}\label{sec:inverseradialmap}

We now need to compute the convergence of the optimal map $T_{\mu}$. The results are summarized in the following proposition, which shows fast asymptotic convergence, especially in the case where $\tilde{f}_{\mu}$ is essentially bounded.

\begin{proposition}\label{prop:inversemapradial}
Fix $y \in \text{int}(\Omega)$. Then, let $x$ and $x_R$ be the points for which $T_{\nu}(x) = y$ and $T_{\nu;R}(x_R) = y$. Let the cost function $c(x,y) = h(\left\Vert x - y \right\Vert)$, where $h(z)$ is increasing, differentiable and strictly convex with $h(0) = 0$ and let there exist $m, M \in \mathbb{R}$ such that $M\geq m> 0$ and $0<m \leq \tilde{f}_{\mu}$, $\Vert \tilde{f}_{\mu} \Vert_{L^{\infty}} \leq M$ and $\tilde{f}_{\nu}>0$. Then, there exists a monotonically increasing function $q: [0, \frac{1}{2} \text{diam}(\Omega)] \rightarrow [0, \infty)$ with $q(t)>0$ for $t>0$ such that:
\begin{equation}
\left\vert x - x_R \right\vert \leq \frac{M \vert y \vert^n}{m(q(\vert y \vert)) \vert \mathbb{S}^{n-1} \vert q\left( \frac{\vert y \vert}{2} \right)^{n-1}} (1 - F_{\nu}(R)).
\end{equation}
If the essential upper bound on $\tilde{f}_{\mu}$ does not exist, then letting $\omega$ be the modulus of continuity of $F_{\mu}$, we have:
\begin{equation}
 \left\vert x - x_{R} \right\vert \leq \frac{1}{m(q(\vert y \vert)) \vert \mathbb{S}^{n-1} \vert q\left( \frac{\vert y \vert}{2} \right)^{n-1}} \omega(\vert y \vert (1 - F_{\nu}(R))).
\end{equation}
\end{proposition}

\begin{proof}
By Theorem~\ref{thm:radial}, we have that:
\begin{equation}
\vert x - x_R \vert = \left\vert F_{\nu}^{-1} \circ F_{\mu} (\vert y \vert) - F_{\nu}^{-1} \circ F_{\mu}(F_{\nu}(R) \vert y \vert) \right\vert.
\end{equation}
let $\tilde{f}_{\nu}(t)\geq m >0$ on $[0, F_{\nu}^{-1} \circ F_{\mu}(\vert y \vert)]$. Then, if we apply Lemma~\ref{lemma:lipschitzinverse} on $F_{\nu}^{-1}$, and following the kind of argument in Proposition~\ref{prop:radialmap} which requires that $f_{\mu}$ is bounded away from zero, we get that there exists a monotonically increasing function $q: [0, \frac{1}{2} \text{diam}(\Omega)] \rightarrow [0, \infty)$ with $q(t)>0$ for $t>0$ such that:
\begin{equation}
\left\vert F_{\nu}^{-1} \circ F_{\mu} (\vert y \vert) - F_{\nu}^{-1} \circ (F_{\nu}(R) \vert y \vert) \right\vert \leq \frac{1}{m(q(\vert y \vert)) \vert \mathbb{S}^{n-1} \vert q(F_{\nu}(R) \vert y \vert)^{n-1}} \left\vert F_{\mu}(\vert y \vert) - F_{\mu}(F_{\nu}(R) \vert y \vert) \right\vert.
\end{equation}
Let us choose $R \geq R_0$, where $R_0$ is the smallest value for which $F_{\nu}(R_0) \geq \frac{1}{2}$. Then,
\begin{equation}
\left\vert F_{\nu}^{-1} \circ F_{\mu} (\vert y \vert) - F_{\nu}^{-1} \circ (F_{\nu}(R) \vert y \vert) \right\vert \leq \frac{1}{m(q(\vert y \vert)) \vert \mathbb{S}^{n-1} \vert q\left( \frac{\vert y \vert}{2} \right)^{n-1}} \left\vert F_{\mu}(\vert y \vert) - F_{\mu}(F_{\nu}(R) \vert y \vert) \right\vert.
\end{equation}
Since a continuous and bounded monotone function is uniformly continuous, we have that there exists a modulus of continuity $\omega :[0, \infty) \rightarrow [0, \infty)$ such that:
\begin{equation}\label{eq:finalineq}
 \left\vert F_{\nu}^{-1} \circ F_{\mu}(\vert y \vert) - F_{\nu}^{-1} \circ F_{\mu}(F_{\nu}(R) \vert y \vert) \right\vert \leq \frac{1}{m(q(\vert y \vert)) \vert \mathbb{S}^{n-1} \vert q\left( \frac{\vert y \vert}{2} \right)^{n-1}} \omega(\vert y \vert (1 - F_{\nu}(R))).
\end{equation}
From the fundamental theorem of calculus and the following estimate, we have that $F_{\mu}$ is locally Lipschitz if and only if the integrand $\tilde{f}_{\mu}(r)r^{n-1}$ is locally essentially bounded which happens if and only if $\tilde{f}_{\mu}$ is essentially bounded from above and the estimate is given by:
\begin{align}
\vert F_{\mu}(r_2) - F_{\mu}(r_1) \vert &= \int_{r_1}^{r_2} \tilde{f}_{\mu}(r) r^{n-1} dr \\
&\leq M \max \{r_1^{n-1}, r_2^{n-1} \} \vert r_2 - r_1 \vert.
\end{align}
Thus, applying this to Equation~\eqref{eq:finalineq}, we get:
\begin{equation}
\left\vert x - x_R \right\vert \leq \frac{M \vert y \vert^n}{m(q(\vert y \vert)) \vert \mathbb{S}^{n-1} \vert q\left( \frac{\vert y \vert}{2} \right)^{n-1}} (1 - F_{\nu}(R)).
\end{equation}
\end{proof}

\subsection{Convergence for the Inverse Radial Potential Function $\phi_{\nu; R}$}\label{sec:radialBrenier}

Despite the relative simplicity of the expression for the radial potential function, the main challenge in obtaining convergence rates for the potential function is that there does not, generally, exist a modulus of continuity for the cost function $h$ for many cost functions of interest, including $h(t) = t$ or $h(t) = t^2$, since these functions are unbounded. With this in mind, we simplify the assumptions so that we can obtain explicit rates.

\begin{proposition}
Let $R \geq R_0$ where $R_0$ is the smallest value that satisfies $F_{\nu}(R_0) \geq 1/2$ and let $R_1$ be such that $\vert \Vert x - T_{R}(x) \Vert - \Vert x - T(x) \Vert \vert \leq \epsilon$ for all $R \geq R_1$. Let the cost function $c(x,y) = h(\left\Vert x - y \right\Vert)$, where $h(z)$ is increasing, differentiable almost everywhere and strictly convex with $h(0) = 0$. Let there exist an $M \geq 0$ and an integer $k \geq 0$ such that $\Vert h'(r) \Vert_{L^{\infty}} \leq \frac{M}{2^{k}} r^{2^{k}-1}$. Then, for all $R \geq \max \{R_0, R_1\}$, we have:
\begin{align}
\vert \phi_{\nu;R}(x) - \phi_{\nu}(x) \vert \leq C(k, \text{diam}(\Omega), M, m, n, \Vert x \Vert) \Vert x \Vert^{2^{k}-1} (1-F_{\nu}(R)),
\end{align}
where $C(\mu, k, \text{diam}(\Omega))$ is given in Inequality~\eqref{eq:radialpotential}.
\end{proposition}

\begin{proof}
We have:
\begin{align}
\vert \phi_{\nu}(x) - \phi_{\nu;R}(x) \vert &= \left\vert \int_{\Vert x - T_{\nu}(x) \Vert}^{\Vert x - T_{\nu;R}(x) \Vert} h'(r) dr \right\vert \\
&\leq \left\vert \int_{\Vert x - T_{\nu}(x) \Vert}^{\Vert x - T_{\nu;R}(x) \Vert} \vert h'(r) \vert dr \right\vert, \\
&= M \left\vert \Vert x - T_{\nu}(x) \Vert^{2^{k}} - \Vert x - T_{\nu;R}(x) \Vert^{2^{k}} \right\vert \\
&= M \left\vert \Vert x - T_{\nu}(x) \Vert - \Vert x - T_{\nu;R}(x) \Vert \right\vert \prod_{i=0}^{k-1} \left( \Vert x - T_{\nu}(x) \Vert^{2^{i}} + \Vert x - T_{\nu;R}(x) \Vert^{2^{i}} \right) 
\end{align}
Since $T_{\nu;R}(x) \rightarrow T_{\nu}(x)$, we can find a $R_1>0$ such that for a fixed $\epsilon>0$ we have $\vert \Vert x - T_{\nu;R}(x) \Vert - \Vert x - T_{\nu}(x) \Vert \vert \leq \epsilon$ for all $R \geq R_1$. Then, denoting $\gamma:= \Vert x \Vert + \frac{1}{2} \text{diam}(\Omega) + \epsilon$, we see that for $R \geq R_1$, we have that $\max \{\Vert x - T_{\nu}(x) \Vert, \Vert x - T_{\nu;R}(x) \Vert \} \leq \gamma$. Therefore,
\begin{align}\label{eq:radialpotential}
\frac{1}{2^{k}} \left\vert \Vert x - T_{\nu}(x) \Vert - \Vert x - T_{\nu;R}(x) \Vert \right\vert &\prod_{i=0}^{k-1} \left( \Vert x - T_{\nu}(x) \Vert^{2^{i}} + \Vert x - T_{\nu;R}(x) \Vert^{2^{i}} \right) \\
&\leq M \Vert T_{\nu}(x) - T_{\nu;R}(x) \Vert  \prod_{i=0}^{k-1} 2 \gamma^{2^{i}} \\
&= 2M \gamma^{2^{k}-1} \Vert T_{\nu}(x) - T_{\nu;R}(x) \Vert \\
&\leq  2 M C(m, n, \Vert x \Vert) \gamma^{2^{k}-1} (1-F_{\nu}(R)) \\
=& C(k, \text{diam}(\Omega), M, m, n, \Vert x \Vert) (1-F_{\nu}(R)),
\end{align}
where the first inequality follows from an application of the reverse triangle inequality.
\end{proof}

We have, as expected, degradation of the convergence rate on the boundary of $B_{R}(0)$.

\subsection{Convergence Rates for Radial Potential Function $\phi_{\mu; R}$}\label{sec:inverseradialBrenier}

The convergence rates of the potential functions and their inverses have the same bound, except possibly on $\partial \Omega$ assuming a coercivity condition on the cost function.

\begin{proposition}
Let the cost function $c(x,y) = h(\left\Vert x - y \right\Vert)$, where $h(z)$ is increasing, differentiable almost everywhere and strictly convex with $h(0) = 0$ and $h'(r) \rightarrow \infty$ as $r \rightarrow \infty$. Let there exist an $M \geq 0$ and an integer $k \geq 0$ such that $\Vert h'(r) \Vert_{L^{\infty}} \leq \frac{M}{2^{k}} r^{2^{k}-1}$. For $y \in \text{int}(\Omega)$, we have that there exists a constant $C(\mu, k, \text{diam}(\Omega), \Vert \tilde{x} \Vert)>0$ and $R_0>0$ such that for all $R \geq R_0$, we have:
\begin{equation}
\vert \phi_{\mu; R}(y) - \phi_{\mu}(y) \vert \leq C(\mu, k, \text{diam}(\Omega), \Vert \tilde{x} \Vert) (1-F_{\nu}(R)).
\end{equation}
\end{proposition}

\begin{proof}
We recall the definitions of the $c$-transforms from Ma, Trudinger, and Wang~\cite{MTW} and the fact that the potential functions for the forward and inverse problems are $c$-transforms of each other. Thus, we have:
\begin{align}\label{eq:ctransform}
\phi_{\mu}(y) &= \inf_{x \in \mathbb{R}^n} \{ c(x,y) - \phi_{\nu}(x) \}, \\
\phi_{\mu; R}(y) &= \inf_{x \in \mathbb{R}^n} \{c(x,y) - \phi_{\nu;R}(x) \}.
\end{align}

We will show that for each $y \in \text{int}(\Omega)$, there exists $\tilde{x} \in \mathbb{R}^n$ such that $\phi_{\mu}(y) = c(\tilde{x}, y) - \phi_{\nu}(\tilde{x})$. This will then show that for such $y$, we have:
\begin{align}
\phi{\mu;R}(y) &\leq c(\tilde{x}, y) - \phi_{\nu;R}(\tilde{x}) \\
&= \phi_{\nu}(\tilde{x}) + \phi_{\mu}(y) - \phi_{\nu;R}(\tilde{x}).
\end{align}
and thus,
\begin{equation}
\phi_{\mu; R}(y) - \phi_{\mu}(y) \leq \phi_{\nu}(\tilde{x}) - \phi_{\nu;R}(\tilde{x}).
\end{equation}
Likewise, we will show that for each $y \in \text{int}(\Omega)$, there exists $\tilde{x}_{R} \in \mathbb{R}^n$ such that $\phi_{\mu;R}(y) = c(\tilde{x}_{R}, y) - \phi_{\nu}(\tilde{x}_{R})$. This will then show that:
\begin{align}
\phi_{\mu}(y) &\leq c(\tilde{x}_{R}, y) - \phi_{\nu}(\tilde{x}_{R}) \\
&= \phi_{\nu;R}(\tilde{x}_{R}) + \phi_{\mu;R}(y) - \phi_{\nu}(\tilde{x}_{R}),
\end{align}
and thus,
\begin{equation}
\phi_{\mu}(y) - \phi_{\mu;R}(y) \leq \phi_{\nu;R}(\tilde{x}_{R}) - \phi_{\nu}(\tilde{x}_{R}).
\end{equation}
Therefore, we have:
\begin{multline}\label{eq:bounds}
\vert \phi_{\mu}(y) - \phi_{\mu;R}(y) \vert \\
\leq \max \{ C(k, \text{diam}(\Omega), M, m, n, \Vert \tilde{x} \Vert), C(k, \text{diam}(\Omega), M, m, n, \Vert \tilde{x}_{R} \Vert) \} (1-F_{\nu}(R)).
\end{multline}

For $y \in \text{int}(\Omega)$, we have that $x$ satisfying $\nabla_{x}c(x,y) -\nabla \phi_{\nu}(x) =0$ is a candidate for an infimum. This equation is solvable for every $y \in \text{int}(\Omega)$ since $\nabla_{x}c(x,T_{\nu}(x)) -\nabla \phi_{\nu}(x) =0$ has a solution and $T_{\nu}(x) = F_{\mu}^{-1} \circ F_{\nu}(x) \hat{x} = y$ is solvable for each $y \in \text{int}(\Omega)$ by the assumption that $f_{\mu}(x) \neq 0$. Thus, we will show that the infimum is not attained when $\Vert x \Vert \rightarrow \infty$. Since $\Vert T_{\nu}(x) \Vert \rightarrow \frac{1}{2} \text{diam}(\Omega)$ as $\Vert x \Vert \rightarrow \infty$, that is $T_{\nu}(x)$ ``approaches" the boundary of $\Omega$ as $\Vert x \Vert$ becomes large, and $y \notin \partial \Omega$, we find that as $\Vert x \Vert \rightarrow \infty$, we have that there exists a radius $\tilde{R}$ such that for all $\Vert x \Vert \leq \tilde{R}$ we have that there exists an $\epsilon>0$ such that $\vert \Vert x - y \Vert - \Vert x - T_{\nu}(x) \Vert \vert \geq \epsilon$. Therefore, for $\Vert x \Vert \geq \tilde{R}$ we have
\begin{align}
c(x,y) - \phi_{\nu}(x) &= h(\Vert x - y \Vert) - \int_{0}^{\Vert x - T_{\nu}(x) \Vert} h'(r)dr \\
&= \int_{\Vert x - T_{\nu}(x) \Vert}^{\Vert x - y \Vert} h'(r)dr \\
&\geq \int_{\Vert x - y \Vert - \epsilon}^{\Vert x - y \Vert} h'(r)dr.
\end{align}
Since $h'(r) \rightarrow \infty$ as $r \rightarrow \infty$, we have that as $\Vert x \Vert \rightarrow \infty$ that $c(x,y) - \phi_{\nu}(x) \rightarrow \infty$. Therefore, the infimum cannot occur as $\Vert x \Vert \rightarrow \infty$ and therefore the infimum occurs for $\tilde{x}$ such that $\nabla_{x}c(\tilde{x},y) -\nabla \phi_{\nu}(\tilde{x}) =0$. The exact same reasoning shows that the infimum for $c(x,y) - \phi_{\nu;R}(x)$ occurs for $\tilde{x}_{R}$ satisfying $\nabla_{x}c(\tilde{x}_{R},y) -\nabla \phi_{\nu}(\tilde{x}_{R}) =0$. Hence, we have established Inequality~\eqref{eq:bounds}.

The last step is to estimate $\Vert \tilde{x}_{R} \Vert$. Since $\tilde{x}_{R}$ satisfies $\nabla_{x}c(\tilde{x}_{R},y) -\nabla \phi_{\nu}(\tilde{x}_{R}) =0$ and for each $y \in \text{int}(\Omega)$ we have that there exists an $x$ such that $T_{\nu}(x) = y$, we find that $T_{\nu}^{-1}(y) = \tilde{x}_{R}$. We have established the convergence rate of $\tilde{x}_{R}$ to $\tilde{x}$ in Theorem~\ref{prop:inversemapradial}. Therefore, fixing $\epsilon>0$, there exists an $R_{2} >0$ such that for all $R \geq R_2$, we have that $\Vert \tilde{x} - \tilde{x}_{R} \Vert \leq \epsilon$. Thus, for $R \geq R_2$, we have the bounds

\begin{align}\label{eq:boundsv2}
\vert \phi_{\mu}(y) - \phi_{\mu;R}(y) \vert \leq C(k, \text{diam}(\Omega), M, m, n, \Vert \tilde{x} \Vert) (1-F_{\nu}(R)).
\end{align}
We remark that for $y \notin \Omega$, we get that the infimum in Equation~\eqref{eq:ctransform} is $-\infty$ achieved when $\Vert x \Vert \rightarrow \infty$. For $y \in \partial \Omega$, the above reasoning is inconclusive.
\end{proof}

\subsection{Convergence Rates of Distribution Function for Certain Classes of Distributions $\mu$}\label{sec:radialdistributions}

In our convergence results in this section, all estimates depend on how quickly the distribution function $F_{\nu}(t)$ converges to $1$ as $t \rightarrow \infty$. Explicit rates then depend on this cumulative distribution function, which depends on the probability measure $\nu$. 

In this manuscript, we examine two classes of probability distributions $\nu$. The first class is probability distributions whose largest finite moment is the $p$th moment. The second class of distributions are those that are log-concave, see Definition~\ref{def:logconcave}. As expected, for the former $1 - F_{\nu}(R)$ converges to $0$ at a rate no worse than a constant times $R^{-p}$ and for the latter, the convergence rate is ``exponential" (see the exact rate in Lemma~\ref{lemma:logcon}.

The result for the finite $p$th moment is simple to obtain:

\begin{lemma}\label{lemma:pmoment}
Let $\mu$ be a radially symmetric probability distribution with radially symmetric density function $f_{\nu}$ and $\tilde{f}_{\nu}(\vert x \vert) = f(x)$ and let $p$ denote the largest positive integer such that $\int_{\mathbb{R}^n} \vert x \vert^{p} d\mu(x) =: M_p$ is finite. Then,
\begin{equation}
\frac{1-F_{\nu}(R)}{\vert \mathbb{S}^{n-1} \vert} \leq M_p R^{-p}.
\end{equation}
\end{lemma}

\begin{proof}
If $\mu \in \mathcal{P}^{p}(\mathbb{R}^{n})$, then this implies that:
\begin{equation}
\int_{0}^{\infty} r^{n-1+p} \tilde{f}_{\nu}(r) dr < \infty.
\end{equation}
Thus, for such probability distributions, we get:
\begin{align}
\frac{1-F_{\nu}(R)}{\vert \mathbb{S}^{n-1} \vert} &= \int_{R}^{\infty} r^{n-1} \tilde{f}_{\nu}(r)dr \\
&= \int_{R}^{\infty} r^{-p} r^{n-1+p} \tilde{f}_{\nu}(r)dr \\
&\leq R^{-p} \int_{R}^{\infty} r^{n-1+p} \tilde{f}_{\nu}(r)dr \\
&\leq M_p R^{-p}.
\end{align}
\end{proof}

For log-concave distributions, we can get an exponentially decaying rate.

\begin{lemma}\label{lemma:logcon}
Let $\mu$ be a radially symmetric log-concave probability distribution over $\mathbb{R}^n$, that is the distribution function of $\mu$ is $e^{-\varphi( \vert x \vert)}$ for $x \in \mathbb{R}^n$, where $\varphi$ is a convex function and $\int_{\mathbb{R}^n} e^{-\varphi(x)}dx = 1$. Then, $\varphi$ is a proper convex function and there exist constants $r_0>0$, $y>0$ such that $\varphi(x) \geq y(\vert x \vert - r_0) + \varphi(r_0)$. Furthermore, we have that there exists a constant $C = C(y, r_0, n)$, when $R \geq 1$, given in Equation~\eqref{eq:radiallogconcave}, such that
\begin{equation}
\frac{1-F_{\nu}(R)}{\vert \mathbb{S}^{n-1} \vert} \leq C(y, r_0, n)R^{n-1}e^{-yR}.
\end{equation}
\end{lemma}

\begin{proof} First, we show that $\varphi$ must be a proper convex function. If it is not proper, then we may have $\varphi(r) = +\infty$ for all $r$, in which case $\int_{\mathbb{R}^d} e^{-\varphi(\vert x \vert)}dx = 0$, in contradiction that we have a probability distribution. Otherwise, we have that $\varphi(r) = -\infty$ for some $r$. In this case, we are forced to conclude that only $\varphi(0) = -\infty$, since otherwise the integral of the probability distribution would blow up. But, in order to be convex, $\varphi(r) = + \infty$ for $r>0$ again leading to a contradiction. So, $\varphi$ is proper.

By convexity and radial symmetry, the minimum of $\varphi(r)$ necessarily occurs at $\varphi(0) = \varphi_0$. The function $\varphi(r)$ is monotonically increasing. It cannot be a constant, for then the integral of the probability distribution would be infinite.

Let $\overline{\varphi}(r)$ denote the lower semi-continuous envelope of $\varphi(r)$. Since $\varphi$ may not be differentiable, we have to make our argument using the sub-differential, see Definition~\ref{def:subdifferential}. Since, $\partial \varphi(r) = \{ 0 \}$ for all $r$ iff $\varphi = \text{const}$, there exists a value $r_0$ where $\partial \overline{\varphi}(r)(r_0) \neq \{ 0 \}$ and $\partial \overline{\varphi}(r)(r_0) \neq \{ +\infty \}$. Here, we make this assertion for $\overline{\varphi}(r)$, since it is possible that $\partial \varphi(r) = \{ 0 \}$ for all $r<r_0$ and then $\partial \varphi(r) = \emptyset$ for all $r\geq r_0$, if $\varphi$ happens to not be lower semi-continuous. Let $y \in \overline{\varphi}(r_0)$ such that $y \neq 0$ and $y < \infty$. Note that since $\overline{\varphi}$ is monotonically increasing, and since a necessary and sufficient condition for the minimum of $\varphi$ is that $0 \in \varphi(r_0)$, we have that $y>0$. This choice of $y$ is possible, since $\partial \overline{\varphi}(r_0) \neq \{ 0\}$ and subdifferentials at a point may include infinity but cannot be equal to the set $\{ +\infty \}$. Given this choice, we have:
\begin{equation}
\varphi(\vert x \vert) \geq y(\vert x \vert - r_0)+ \varphi(r_0),
\end{equation}
since the tangent line of a convex function lies below the graph of the convex function. Now, we use this lower bound on $\varphi$ to estimate the following:
\begin{align}
\frac{1 - F_{\nu}(R)}{\vert \mathbb{S}^{n-1} \vert} &= \int_{R}^{\infty} r^{n-1}e^{-\varphi(r)}dr \\
&\leq e^{yr_0-\varphi(r_0)} \int_{R}^{\infty} r^{n-1}e^{-yr}dr.
\end{align}
At this point we have a straightforward computation which begins by repeated application of integration by parts. After performing integration by parts we could stop, but choose to simplify the expressions in order to easily see the dependence on $R$. We compute:
\begin{align}\label{eq:radiallogconcave}
\frac{1 - F_{\nu}(R)}{\vert \mathbb{S}^{n-1} \vert} &\leq e^{yr_0-\varphi(r_0) - yR} \left( \frac{R^{n-1}}{y^n} +\sum_{k=2}^{n} \frac{(n-1) \dots (n-k+1)}{y^{k}} R^{n-k} \right) \\
&\leq e^{yr_0-\varphi(r_0) - yR} R^{n-1} \left( \frac{1}{y^n} +\sum_{k=2}^{n} \frac{(n-1) \dots (n-k+1)}{y^{k}} \right) \\
&=\frac{e^{yr_0-\varphi(r_0) - yR} R^{n-1}}{y^n} \left( 1 + (n-1)! \sum_{k=0}^{n-2}\frac{y^k}{k!} \right) \\
&\leq \frac{e^{yr_0-\varphi(r_0) - yR} R^{n-1}}{y^n} \left( 1 + (n-1)! \left(e^y - \frac{y^{n-1}}{(n-1)!} \right) \right) \\
&=C(y, r_0 ,n) R^{n-1}e^{-yR},
\end{align}
where the second line follows from using $R \geq 1$, the third line from algebraic manipulations, and the fourth line from Taylor's theorem applied to the function $e^y$ expanded about $y = 0$.
\end{proof}

\section{Convergence Result for General Case}\label{sec:general}
In this section, which constitutes the main results of this manuscript, we are concerned with the case where the source and target measures are not assumed to be radially symmetric. We establish convergence results in this case only for the cost function $c(x,y) = \frac{1}{2}\Vert x - y \Vert^2$. In Section~\ref{sec:L2}, we introduce the setup we are considering, the fundamental $L^2$ estimate from Delalande and M\`{e}rigot~\cite{merigot1}, and the stability in the Wasserstein-2 distance of the cutoff problem. In Section~\ref{sec:W1} we show how the distance $W_1(\nu, \nu_{R})$ is estimated from above and use this to show the $L^2$ convergence rate for the cutoff method for measures $\nu$ that have finite $p$th moment and for the case when $\nu$ is a log-concave measure. The remaining part of this section is written without explicit reference to the cutoff method. This is done to show that once $L^1$ rates have been established, for whichever way $\nu$ is approximated, an entire list of very strong results follow. Notice that while our objective was to use these for studying the case where the target measure has unbounded support, these results can immediately be extended to the case where the target measure $\nu$ has compact support, as the $L^2$ bounds by Delalande and M\`{e}rigot apply to this case as well. The only difference would be in the actual explicit rates we would obtain.

In Section~\ref{sec:pointwise}, we show that the Brenier potentials converge pointwise. This then leads to showing that the optimal map converges almost everywhere. In Section~\ref{sec:explicituniform}, the main result of this section (Theorem~\ref{thm:explicitrates1}), we show that uniform explicit rates can be determined for a case where the Brenier potentials are H\"{o}lder continuous. In Section~\ref{sec:inversebrenier} we show that the same convergence for the Brenier potential more-or-less carries over to the inverse Brenier potential. And, finally, in Section~\ref{sec:inversemap}, we show that the inverse map converges almost everywhere.

\subsection{$L^2$ Convergence Rates}\label{sec:L2}

Let $\Omega \subset \mathbb{R}^n$ be a compact convex set and let $c(x,y) = \frac{1}{2}\Vert x - y \Vert^2$ and $\text{spt}(\mu) = \Omega$. Let $\text{spt}(\nu) \subset \mathbb{R}^n$ be unbounded. We further assume that $\mu$ and $\nu$ are absolutely continuous with respect to the Lebesgue measure and that $\nu \in \mathcal{P}_2(\mathbb{R}^n)$, that is, there exists a finite constant $M_2>0$ such that $\int_{\mathbb{R}^n} \vert x \vert^2 d\nu(x) \leq M < \infty$. The last assumption, in particular, guarantees the finiteness of $W_2(\mu, \nu)$. We also introduce the following property of sets:
\begin{definition}\label{def:epsilonball}
Given an closed set $A \subset \mathbb{R}^n$, we say that $A$ has the inner $\epsilon$-ball property if for a fixed $\epsilon>0$ there exists a set $A '$ such that $A = \cup_{x \in A '} \overline{B_{\epsilon}(x)}$.
\end{definition}
Notice that $\Omega$ having the $\epsilon$-ball property for some $\epsilon$ is equivalent to $\Omega$ being the Minkowski sum of some set $S \subset \mathbb{R}^n$ and $\overline{B_{\epsilon}(0)}$, that is $\Omega = S \oplus \overline{B_{\epsilon}(0)}$. In Section~\ref{sec:explicituniform}, we will derive explicit pointwise convergence rates for the case where $p>n$ and $p \geq 4$, where $p$ is the highest finite moment of $\nu$. In order to derive such rates, we will impose that there exists an $\epsilon>0$ such that $\Omega$ satisfies the inner $\epsilon$-ball property. We will also establish analogous results for the case where $\Omega$ is rectangular.

The convergence results in this section are made possible by the important $L^2$ stability results established in Delalande and M\'{e}rigot~\cite{merigot1}. These will lead to explicit pointwise convergence rates that we establish in this section, because we are approximating the target measure $\nu$ by the cutoff approximation. In the case that the $L^2$ convergence rates of Delalande and M\'{e}rigot do not apply we still provide theorems guaranteeing pointwise convergence in anticipation that such $L^2$ stability results can and will be made explicit in the future.

For this section, we settle the choice of the arbitrary constant for the Brenier potential $\phi_{\mu}$ by requiring that $\int_{\Omega} \phi_{\mu}(x)dx = 0$. The main result from Delalande and M\`{e}rigot~\cite{merigot1} that we will use is the following $L^2$ stability result, adapted to the context of the present manuscript:
\begin{theorem}[Delalande and M\`{e}rigot]\label{thm:delalandemerigot}
Let $p$ denote the highest finite moment of $\mu$ and $\nu$, and let the value of this moment be denoted as $M_p$ and consider $\mu$ and $\nu$ for which $p>n$ and $p \geq 4$. Let the density function of $\mu$ be bounded below and above by constants. Then,
\begin{align}\label{eq:L2convergence}
\left\Vert T_{\mu} - T_{\mu;R} \right\Vert_{L^2 \left( \mu \right)} &\leq C_{n, p, \overline{\Omega}, \mu, M_{p}}  W_{1}(\nu, \nu_{R})^{\frac{p}{6p+16n}}, \\
\left\Vert \phi_{\mu}- \phi_{\mu;R} \right\Vert_{L^2 \left( \mu \right)} &\leq C_{n, p, \overline{\Omega}, \mu, M_{p}}  W_{1}(\nu, \nu_{R})^{1/2}.
\end{align}
\end{theorem}

Even without the $L^2$ bounds, we can show that $W_2(\mu, \nu_{R}) \rightarrow W_2(\mu, \nu)$ only requiring the very weak assumption that $\nu \in \mathcal{P}_{2}(\mathbb{R}^n)$ which assures the finiteness of the Wasserstein-2 distance for this problem.

\begin{proposition}
For $p \geq 2$, we have that:
\begin{equation}
W_2(\mu, \nu_{R}) \rightarrow W_2(\mu, \nu),
\end{equation}
where $W_2(\mu, \nu)$ is the Wasserstein-2 distance between $\mu$ and $\nu$ supported on $\mathbb{R}^n$.
\end{proposition}
\begin{proof}
The proof follows from Theorem 6.9 in Villani~\cite{villani2}, once we have shown that $\nu_{R} \rightarrow \nu$ weakly in $\mathcal{P}_{2}(\mathbb{R}^n)$ and that $\int_{\mathbb{R}^n} \frac{1}{2} \Vert x_0 - x \Vert^2 d\nu(x) <\infty$ for any fixed $x_0 \in \Omega$. For any $g \in C_{b}(\mathbb{R}^n)$, we have that:
\begin{equation}\label{eq:weakconv}
\int_{\mathbb{R}^n} g(x) d\nu_{R}(x) = \frac{1}{\nu(B_{R}(0))} \int_{\mathbb{R}^n} g(x) d\nu(x) \rightarrow \int_{\mathbb{R}^n} g(x) d\nu(x),
\end{equation}
since $\nu(B_{R}(0)) \rightarrow 1$ as $R \rightarrow \infty$. Therefore, $\nu_R \rightarrow \nu$ weakly. Fixing $x_0 \in \Omega$, we compute:
\begin{align}
\int_{\mathbb{R}^n} \frac{1}{2} \Vert x_0 - x \Vert^2 d\nu(x) &= \int_{\mathbb{R}^n} \frac{1}{2} \Vert x_0 \Vert^2 d\nu(x)  + \int_{\mathbb{R}^n} \frac{1}{2} \Vert x \Vert^2 d\nu(x)  \\
 &+ \int_{\mathbb{R}^n } x_0 \cdot x d\nu(x) \\
&\leq \frac{1}{2} \Vert x_0 \Vert^2 + \int_{\mathbb{R}^n} \frac{1}{2} \Vert x \Vert^2 d\nu(x) + \int_{\mathbb{R}^{n}} \Vert x_0 \Vert \Vert x \Vert d\nu(x) \\
&\leq \frac{1}{2} \Vert x_0 \Vert^2 + \int_{\mathbb{R}^n} \frac{1}{2} \Vert x \Vert^2 d\nu(x) + \int_{\vert x \vert \leq 1} \Vert x_0 \Vert d\nu(x) \\
&+ \int_{\vert x \vert \geq 1} \Vert x_0 \Vert \Vert x \Vert^2 d\nu(x) \\
&<\infty.
\end{align}
\end{proof}

\subsection{Estimating $W_{1}(\nu, \nu_{R})$}\label{sec:W1}

We estimate the quantity $W_{1}(\nu, \nu_{R})$ in two cases. The first is under the assumption that $\nu$ has a largest finite moment $p$. The second assumption is that $\nu$ is a log-concave distribution, which will lead to exponential convergence of $W_{1}(\nu, \nu_{R})$. First, we define a quantity which will serve as the natural generalization of the cumulative distribution function and which will, in the cases considered in this manuscript, control the rate of convergence of $W_{1}(\nu, \nu_R)$ to zero. We begin with the following lemma, which covers the most general case, where $\nu$ is simply a probability distribution in $\mathbb{R}^n$.
\begin{lemma}\label{lemma:firstbound}
Let $\nu \in \mathcal{P}(\mathbb{R}^n)$. Then, the following bounds the $W_1$ distance between $\nu$ and $\nu_R$:
\begin{equation}\label{eq:generalbound}
W_1(\nu, \nu_R) \leq \int_{\mathbb{R}^n \setminus B_{R}(0)} \vert x \vert d\nu(x) + R(1-\nu(B_{R}(0))).
\end{equation}
\end{lemma}
\begin{proof}
For any $S$ satisfying $S_{\#}\nu = \nu_R$, we have the bound $W_1(\nu, \nu_{R}) \leq \int_{\mathbb{R}^n} \vert x - S(x) \vert d\nu(x)$ . Now, we construct a specific pushforward map $S$. Define the probability measure to be the the unique measure $\tilde{\nu}$ satisfying $d\tilde{\nu}(x) := d\nu/(1-\nu(B_{R}(0)))$ for $x \in \mathbb{R}^d \setminus B_{R}(0)$ and $d \tilde{\nu}(x) = 0$ for $x \in B_{R}(0)$. Also, define $\tilde{\nu}_{R}$ where $d \tilde{\nu}_{R}(x) = d\nu_R(x) - d\nu(x)$ for $x \in B_{R}(0)$ and $d\tilde{\nu}_{R} = 0$ for $x \in \mathbb{R}^d \setminus B_{R}(0)$. Let $M(x)$ be any measure preserving mapping between $\tilde{\nu}$ and $\tilde{\nu}_{R}$, for example an optimal transport map, which is guaranteed to exists since $\nu$ has a density function. Then, we define the map $S(x)$ as follows:
\begin{equation}
S(x):= \begin{cases}
x, & x \in B_{R}(0), \\
M(x), & x \in \mathbb{R}^n \setminus B_{R}(0).
\end{cases}
\end{equation}
Notice that $S(x)$ satisfies $S_{\#}\nu = \nu_R$. With these definitions, we obtain the simple estimate:

\begin{align}
W_{1}(\nu, \nu_{R}) &\leq \int_{\mathbb{R}^d} \vert x-S(x) \vert d\nu(x) \\
&= \int_{\mathbb{R^d} \setminus B_{R}(0)} \vert x- M(x)\vert d\nu(x) \\
&\leq \int_{\mathbb{R^d} \setminus B_{R}(0)} (\left\vert x \right\vert+R)d\nu(x) \\
&= \int_{\mathbb{R^d} \setminus B_{R}(0)} \left\vert x \right\vert d\nu(x) + R(1-\nu(B_{R}(0)).
\end{align}
\end{proof}

Defining the $n$-cube as $C_{1}(0) := \{x \in \mathbb{R}^n : -1 \leq x_1 \leq 1, \dots, -1 \leq x_n \leq 1 \}$ and
\begin{equation}
d \nu_{R_{C}} := \begin{cases}
d \nu(x) / \nu(RC_1(0)), & \text{if} \ x \in RC_1(0) \\
0, & \text{otherwise},
\end{cases}
\end{equation}
then we have the following result:
\begin{lemma}
Let $\nu \in \mathcal{P}(\mathbb{R}^n)$. Then, the following bounds the $W_1$ distance between $\nu$ and $\nu_{R_{C}}$:
\begin{equation}\label{eq:generalbound}
W_1(\nu, \nu_{R_{C}}) \leq \int_{\mathbb{R}^n \setminus RC_{1}(0)} \vert x \vert d\nu(x) + \sqrt{n} R(1-\nu(RC_1(0))).
\end{equation}
\end{lemma}
The proof follows exactly in the same way as the proof for Lemma~\ref{lemma:firstbound}. In this section, we have assumed that $\nu \in \mathcal{P}_2(\mathbb{R}^n)$. We thus generalize the result from Lemma~\ref{lemma:firstbound} to the case where $\nu$ has a finite $p$th moment.

\begin{proposition}\label{lemma:bounds1}
Let $p$ denote the highest finite moment of $\nu$. Then, we have:
\begin{equation}
W_1 (\nu, \nu_{R}) \leq 2M_p R^{1 - p},
\end{equation}
where $M_p$ is the $p$th moment of $\nu$.
\end{proposition}

\begin{proof}
Since the $p$th moment is bounded, a direct application of H\"{o}lder's inequality yields:
\begin{align}
\int_{\mathbb{R}^d \setminus B_{R}(0)} \left\vert x \right\vert d\nu(x) &\leq (1-\nu(B_{R}(0)))^{1/q}\left( \int_{\mathbb{R}^d \setminus B_{R}(0)} \left\vert x \right\vert^p d\nu(x) \right)^{1/p} \\
&\leq (M_p)^{1/p} (1-\nu(B_{R}(0)))^{1/q},
\end{align}
where $q$ is the conjugate of $p$. This bound can, of course be tightened, since the finiteness of the integral $\int_{\mathbb{R}^d} \left\vert x \right\vert^p d\mu(x)$ implies $\tilde{M}_{p}(R) := \int_{\mathbb{R}^d \setminus B_{R}(0)} \left\vert x \right\vert^p d\mu(x) \rightarrow 0$ as $R \rightarrow \infty$. Thus, by Lemma~\ref{lemma:firstbound}, we have the estimate:
\begin{equation}
W_1(\nu, \nu_{R}) \leq R(1 - \nu(B_{R}(0))) + (M_p)^{1/p}(1 - \nu(B_{R}(0)))^{1/q}.
\end{equation}
If there exists a constant $C(\nu)$ such that for large enough $R$ we have $(1 - \nu(B_{R}(0))) \leq C(\nu) R^{-p}$, then the term $(1-\nu(B_{R}(0)))^{1/q}$ dominates and we get:
\begin{align}\label{eq:rp}
W_1(\nu, \nu_{R}) &\leq (1 - \nu(B_{R}(0)))^{1/q} \left((M_p)^{1/p} + R(1-\nu(B_{R}(0)))^{1/p} \right) \\
&= (1 - \nu(B_{R}(0)))^{1/q} \left((M_p)^{1/p} + (C(\nu))^{1/p} \right).
\end{align}
We now apply Markov's inequality to show that, in fact, $1 - \nu(B_{R}(0)) \leq M_{p} R^{-p}$, i.e. Inequality~\eqref{eq:rp} holds. Let the random variable $X$ have the law $\mu$. Then,
\begin{align}\label{eq:Markov}
\int_{\mathbb{R}^n \setminus B_{R}(0)} d\mu(x) &= \mathbb{P}[\vert X \vert \geq R], \\
&= \mathbb{P}[\vert X \vert^{p} \geq R^{p}] \\
&\leq \mathbb{E}[X^p]R^{-p} \\
&= M_p R^{-p}.
\end{align}
Thus, we see that
\begin{align}
W_{1}(\nu, \nu_{R}) &\leq 2M_p^{1/p} M_{p}^{1/q} R^{-p/q} \\
&= 2M_p R^{1 - p}.
\end{align}
\end{proof}

We also have the following result for the $n$-cube:
\begin{proposition}
Let $p$ denote the highest finite moment of $\nu$. Then, we have:
\begin{equation}
W_1 (\nu, \nu_{R_{C}}) \leq 2M_pR^{1 - p}.
\end{equation}
\end{proposition}
\begin{proof}
The proof follows in the same way as the proof for Proposition~\ref{lemma:bounds1}, \textit{mutatis mutandis} and with the additional line in Inequality~\eqref{eq:Markov} that $\nu(\mathbb{R}^n \setminus RC_1(0)) \leq \nu(\mathbb{R}^n \setminus B_{R}(0))$.
\end{proof}

For log-concave distributions we can also produce a bound and also tighten the constant $C(\nu)$. Of course, for such distributions, Inequality~\eqref{eq:rp} applies.

\begin{proposition}\label{prop:logconcave}
Let $\nu$ be a log-concave distribution, i.e. $f_{\nu}(x) = e^{-\varphi(x)}$ for a convex function $\varphi(x)$ and let $R \geq 1$. Then, there exist constants $C = C(\nu, n)$ and $a>0$ such that:
\begin{equation}
W_1(\nu, \nu_{R}) \leq C(\nu, n) R^n e^{-aR},
\end{equation}
where $C(\nu, n)$ is given in Equation~\eqref{eq:fastrate}.
\end{proposition}

\begin{proof}
In Parts A-D, we prove that there exist constants $a>0$ and $b \in \mathbb{R}$ such that $\varphi(x) \geq a \vert x \vert + b$. We use the fact that $\varphi$ is convex and that $\int_{\mathbb{R}^n} e^{-\varphi(x)}dx = 1$. Most of the difficulty of the proof is contained in showing this (Parts A-D). Once we have this bound, we perform a routine computation to get the Wasserstein-1 distance between $\mu$ and $\mu_{R}$ (Part G).

Part A) We first show that $\varphi(x)$ must be a proper convex function. Suppose $\varphi(x)$ is not proper. This means that i) $\varphi(x) = +\infty$ for all points, in which case $e^{-\varphi(x)} = 0$ for all $x\in \mathbb{R}^n$ which contradicts the fact that $\int_{\mathbb{R}^n} e^{-\varphi(x)}dx = 1$ or ii) $\varphi(x) = -\infty$ for some $x \in \mathbb{R}^n$. Let $E$ denote the set where $\varphi(x) = -\infty$. Then $E$ must be a convex set of Lebesgue measure zero, which is a subset of a hyperplane. Then $\partial E$ is a measurable set since it is the boundary of a convex set, and it has measure zero. The set $\partial E$ is the only set on which $\varphi$ can take finite values. Because of this, then we are forced to conclude that $\int_{\mathbb{R}^n} e^{-\varphi(x)}dx = 0$, which, again, is a contradiction. For a similar reason, $\varphi$ cannot only be finite on a set of measure zero.

Part B) Now that we know $\varphi(x)$ must be proper, let us assume that $\varphi(x)$ has a minimum on $\mathbb{R}^n$ at, say, $x_0 \in \mathbb{R}^n$. We will demonstrate that if the affine dimension of $\text{dom}(\varphi)$ is equal to $n$, then the minima of $\varphi$ must exist on a compact set. Also, by the integrability condition, we cannot have the the affine dimension of $\text{dom}(\varphi)$ be less than $n$.

After an affine coordinate transformation, we assume $\varphi(0) = 0$. If $\varphi(x) = 0$ for some $x \neq 0$, then $\varphi$ must be equal to zero along the line connecting $x$ and $0$, since $\varphi$ is convex. Therefore, if $\varphi(x)$ takes the value zero arbitrarily far away from the origin, it takes the value zero on the ray emanating from the origin. By another affine coordinate transformation, we take this ray to be in the direction of the positive first coordinate axis. Thus, by the choice of our coordinate system, $\varphi(x)=0$ for any $x$ such that $x_1 \geq 0$ and $x_2 = \dots = x_n = 0$. Since in Part A) we proved that $\varphi$ cannot only be finite on a set of measure zero, and $\text{dom}(\varphi)$ is convex, then the affine dimension of $\text{dom}(\varphi)$ must be equal to $n$. Thus, $\text{int}(\text{dom}(\varphi)) \neq \emptyset$. Let $\tilde{x} \in \text{int}(\text{dom}(\varphi))$. Then, there exists a ball $B_{\epsilon}(\tilde{x}) \subset \text{int}(\text{dom}(\varphi))$, and therefore $\varphi(y) <\infty$ for every $y \in B_{\epsilon}(\tilde{x})$. Since $e_1$ is a direction of recession of $\text{dom}(\varphi)$, then $y + te_1 \in \text{dom}(\varphi)$ for $t\geq 0$ and every $y \in B_{\epsilon}(\tilde{x})$. Furthermore, $(e_1, 0) \subset \mathbb{R}^{n+1}$ is a direction of recession of $\text{epi}(\varphi)$ and therefore, $\varphi(y + t e_1) \leq \varphi(y)$ for every $t \geq 0$.  Thus,
\begin{align}
\int_{\mathbb{R}^n} e^{-\varphi(x)}dx &\geq \sum_{k=1}^{\infty} \int_{B_{\epsilon}(\tilde{x} + 2 \epsilon k)} e^{-\varphi(x)}dx \\
&=\sum_{k=1}^{\infty} \int_{B_{\epsilon}(\tilde{x})} e^{-\varphi(x)}dx = \infty.
\end{align}

Part C) Suppose $\varphi(x)$ has no minimum. By an affine change of coordinates, let $\varphi(0) = 0$. Then, $\varphi(x) \rightarrow -\infty$ along a ray $R = \{ay: a\geq 0, ,y \in \mathbb{R}^n\}$ emanating from the origin. Because $\varphi(x)$ is convex, it must be convex along the ray $R$. Since $\varphi(0) = 0$ and $\varphi$ is convex (and therefore monotonically decreasing along the ray $R$), then $\varphi(x) \leq 0$ for $x \in R$.

Then, we define a new function $\tilde{\varphi}(x) = \varphi(x)$ for $x$ such that $\varphi(x) \geq 0$ and $\tilde{\varphi}(x) = 0$ for $x$ such that $\varphi(x) < 0$. Then, $\tilde{\varphi}$ is convex and has a minimum. But, then $\int_{\mathbb{R}^n} e^{-\varphi(x)}dx \geq \int_{\mathbb{R}^n} e^{-\tilde{\varphi}(x)}dx = \infty$, by applying Part B to the function $\tilde{\varphi}(x)$.

Part D) We have thus showed that $\varphi$ must have a minimum. If there do not exist a constant $a>0$ and a vector $b \in \mathbb{R}$ such that $\varphi(x) \geq b + a \vert x \vert$, this must be because $\varphi$ is constant in a ray emanating from the minimum point $x_0$. Without loss of generality, suppose that the minimum of $\varphi$ occurs at the origin. Fix $b  < \varphi(0)$ and $a>0$. Let $S = \{ x: b + a \vert x \vert \leq \varphi(x) \}$. Then, if there is a point $y \in \mathbb{R}^n$ such that $\varphi(y) < b + a \vert y \vert$. Thus, we can find a new value $a'$ that satisfies $a' \leq \frac{\varphi(y) - b}{\vert y \vert}$ and will guarantee that $\forall x \in S$, we have $\varphi(x) \geq b + a' \vert x \vert$. We thus see that the existence of a uniform $a$ that guarantees that $\varphi(x) \geq b + a \vert x \vert$ for all $x \in \mathbb{R}^n$ is equivalent to guaranteeing that $\alpha = \inf_{y} \frac{\varphi(y) - b}{\vert y \vert} > 0$. Since $b<\varphi(x)$ for all $x \in \mathbb{R}^n$, this is only possible if $\varphi(x)$ grows asymptotically slower than any linear function as $\vert y \vert \rightarrow \infty$. The constant $\alpha = 0$ only if $\varphi$ is constant in some ray emanating from the origin (the minimum). This is because if $\varphi$ were not constant along any ray emanating from the origin, then along any ray emanating from the origin, by the convexity of $\varphi$, there would exist a tangent line with nonzero slope lying below $\varphi$ restricted to the ray. Then, since $\varphi$ lies above the tangent line, it grows at least linearly as $\vert y \vert \rightarrow \infty$. To be more precise, fix $z \in \mathbb{S}^{n-1}$ and define $r_{z} = \{ t z : t \geq 0 \}$. Define $\varphi_{z}(t) := \varphi(t z)$. Then, $\varphi_{z}(t)$ is convex on $[0, \infty)$. If $\varphi_{z}$ is not constant on $[0, \infty)$, then since $\varphi_{z}(t)$ is finite for $t = 0$, there exists a value $y >0$ such that $y \in \partial \varphi_{z}(t_0)$ for some $t_0 \in [0, \infty)$. Therefore, $\varphi_{z}(t) \geq \varphi_{z}(t_0) + y(t - t_0)$ for all $t \in [0, \infty)$. Thus, since the value of $\alpha$ can be found by minimizing along rays and then minimizing over $\mathbb{S}^{n-1}$, which is a compact set, the only way for $\alpha$ to equal zero is if $\varphi$ were constant in some ray emanating from the origin. Since above we have proved that this is not possible, we conclude that there must exist $a>0$ and $b \in \mathbb{R}$ such that $\varphi(x) \geq b + a \vert x \vert$ for all $x \in \mathbb{R}^n$.

Part E) This then shows that there exist $a>0$ and $b \in \mathbb{R}$ such that $\varphi(x) \geq a\left\vert x \right\vert + b$. Thus, for log-concave distributions $\nu$, we can use Lemma~\ref{lemma:firstbound}. First, we compute:
\begin{align}\label{eq:mbn1}
1 - \nu(B_{R}(0)) &= \int_{\mathbb{R}^{n-1} \setminus B_{R}(0)} e^{-\varphi(x)}dx \\
&\leq \int_{\mathbb{R}^{n} \setminus B_{R}(0)} e^{-a \vert x \vert - b}dx \\
&\leq \vert \mathbb{S}^{n-1} \vert e^{-b} \int_{R}^{\infty} r^{n-1}e^{-a r}dr \\
&= \vert \mathbb{S}^{n-1} \vert e^{-b-aR} R^{n-1} \sum_{k=1}^{n} \frac{R^{1-k}}{a^k} \frac{(n-1)!}{(n-k)!}.
\end{align}
Therefore, for $R \geq 1$, we have:
\begin{align}
1 - \nu(B_{R}(0)) \leq \vert \mathbb{S}^{n-1} \vert e^{-b-aR} R^{n-1} \sum_{k=1}^{n} \frac{(n-1)!}{a^k(n-k)!}.
\end{align}
Likewise, we obtain the estimate, for $R \geq 1$:
\begin{equation}\label{eq:mbn2}
\int_{\mathbb{R}^n \setminus B_{R}(0)} \vert x \vert e^{-\varphi(x)}dx \leq \vert \mathbb{S}^{n-1} \vert e^{-b - aR} R^n \sum_{k=1}^{n+1} \frac{n!}{a^k (n - k+1)!}.
\end{equation}
Therefore, using Inequality~\eqref{eq:generalbound}, we obtain, for $R \geq 1$, that:
\begin{align}\label{eq:fastrate}
W_1(\nu, \nu_R) &\leq \left( \sum_{k=1}^{n+1} \frac{n!}{a^k(n - k + 1)!} + \sum_{k=1}^{n} \frac{(n-1)!}{a^k(n - k)!} \right) \vert \mathbb{S}^{n-1} \vert R^n e^{-b - aR} \\
&= C(a, b, n) R^n e^{-aR}.
\end{align}
\end{proof}

This begins to hint at the success of this cutoff method, since we will expect exponential-type convergence for log-concave distributions. We have the following result for the $n$-cube case:

\begin{proposition}
Let $\nu$ be a log-concave distribution, i.e. $f_{\nu}(x) = e^{-\varphi(x)}$ for a convex function $\varphi(x)$ and let $R \geq 1$. Then, there exist constants $C = C(\nu, n)$ and $a>0$ such that:
\begin{equation}
W_1(\nu, \nu_{R_{C}}) \leq C(\nu, n) R^n e^{-aR},
\end{equation}
where $C(\nu, n)$ is given in Equation~\eqref{eq:fastrate}.
\end{proposition}
\begin{proof}
The only modifications needed from the proof of Proposition~\ref{prop:logconcave} are in Inequality~\eqref{eq:mbn1} where we note that $\nu(\mathbb{R}^n \setminus RC_1(0)) \leq \nu(\mathbb{R}^n \setminus B_{R}(0))$ and in Inequality~\eqref{eq:mbn2} where we note that $\int_{\mathbb{R}^n \setminus RC_1(0)} \vert x \vert e^{-\varphi(x)}dx \leq \int_{\mathbb{R}^n \setminus B_{R}(0)} \vert x \vert e^{-\varphi(x)}dx$. 
\end{proof}

\subsection{Pointwise Convergence of Brenier Potential and Pointwise Almost-Everywhere Convergence of Optimal Mapping}\label{sec:pointwise}

Once we have $L^1$ or $L^2$ rates in terms of $R$ for convergence on a compact set $\Omega$, we can establish pointwise convergence for the Brenier potentials and almost everywhere convergence for the optimal mapping. In fact, we do not need explicit $L^p$ rates, but simply convergence in $L^p$. Since the Wasserstein distance metrizes weak convergence for probability measures, see Theorem 6.9 from Villani~\cite{villani2}, this fact along with Theorem~\ref{thm:delalandemerigot} then implies that as long as $\mu_{R} \rightharpoonup \mu$ weakly, then we have pointwise convergence on any strict open subset of $\Omega$, once we assume the density function of $\mu$ is strictly bounded away from zero. Therefore, the only place we do not necessarily expect pointwise convergence, due to the convexity property of the Brenier potentials, is the boundary $\partial \Omega$.

First, we establish that there are sequences $\{\phi_{\mu;R_{k}} \}_{k}$ such that $\phi_{\mu;R_{k}} \rightarrow \phi_{\mu}$ a.e. as $k \rightarrow \infty$.

\begin{proposition}\label{prop:almosteverywhere}
For probability distributions $\mu$, where the density function $f_{\mu}$ of $\mu$ satisfies $f_{\mu}>0$ on $\Omega$ and $\nu$ such that $\left\Vert \phi_{\mu}- \phi_{\mu;R} \right\Vert_{L^2(\mu)} \rightarrow 0$ as $R \rightarrow 0$, we have that if $R_{k} \rightarrow \infty$ converges fast enough as $k \rightarrow \infty$, then $\phi_{\mu;R_{k}}$ converges to $\phi_{\mu}$ almost everywhere pointwise (and almost uniformly) on $\Omega$.
\end{proposition}

\begin{proof}
First, we bound the $L^1$ norm above by the $L^2$ norm since we are on a compact set using H\"{o}lder's inequality:
\begin{align}
\left\Vert \phi_{\mu} - \phi_{\mu;R} \right\Vert_{L^1(\mu)} &\leq  \sqrt{\vert \Omega \vert} \left\Vert \phi_{\mu} - \phi_{\mu;R} \right\Vert_{L^2(\mu)}, \\
&\leq \sqrt{\vert \Omega \vert} C_{n, p, \overline{\Omega}, \mu, M_{p}}  W_{1}(\nu, \nu_{R})^{1/2}
\end{align}
We now define a sequence $\{R_k\}$ such that $\sum_{k=1}^{\infty} W_1(\nu, \nu_{R_{k}}) <\infty$. With this choice, we see that:
\begin{equation}
\sum_{k=1}^{\infty} \left\Vert \phi_{\mu} - \phi_{\mu;R_{k}} \right\Vert_{L^1(\mu)} <\infty.
\end{equation}
We will rewrite our $L^1$ norm as an expectation over a random variable whose law is $\mu$. We denote $g_k(x) := \vert \phi_{\mu;R_{k}}(x) - \phi_{\mu}(x) \vert$. Note that $g_{k}(X)$ is a non-negative random variable, where $X$ is the random variable whose law is $\mu$. With this notation, we write:
\begin{equation}
\Vert \phi_{\mu;R_{k}}-\phi_{\mu} \Vert_{L^1(\mu)} = \mathbb{E}_{\mu} [g_k(X)].
\end{equation}
Hence, denoting $\mathcal{L}_{\Omega}$ to be the Lebesgue measure restricted to $\Omega$, we find, using Markov's inequality, that for $a>0$:
\begin{equation}
\mathbb{P}_{\mu}(g_k(X) \geq a) \leq \frac{\mathbb{E}_{\mu} [g_k(X)]}{a},
\end{equation}
and thus,
\begin{equation}
\sum_{k=1}^{\infty} \mathbb{P}_{\mu}(g_k(X) \geq a) \leq \sum_{k=1}^{\infty} \frac{\mathbb{E}_{\mu} [g_k(X)]}{a} <\infty.
\end{equation}
Hence, since $\min_{x \in \overline{\Omega}} f >0$, we see:
\begin{equation}
\sum_{k=1}^{\infty} \mathcal{L}_{\Omega} (g_k(x) \geq a) < \infty.
\end{equation}
Applying the Borel-Cantelli lemma, we immediately conclude that:
\begin{equation}
\mathcal{L}_{\Omega} \left( \limsup_{k \rightarrow \infty} (g_k(x) \geq a) \right) = 0,
\end{equation}
for any $a>0$, where $\limsup_{k \rightarrow \infty} A_k = \cap_{k=1}^{\infty} \cup_{j=k}^{\infty} A_j$. The conclusion states that the set of points where $g_k(x) \geq a$ is true infinitely often (for any $a>0$) is a set of measure zero. This is the set where we do not have pointwise convergence to the zero function. Therefore, we conclude that $g_{k}(x) \rightarrow 0$ pointwise almost everywhere. Since $\Omega$ is bounded, by applying Egorov's theorem we get that the convergence is almost uniform on $\Omega$.
\end{proof}
Due to the convexity of $\phi_{\mu;R}$ and $\phi_{\mu}$, pointwise almost everywhere convergence becomes pointwise convergence on any open strict subset of $\Omega$ and uniform convergence on any closed subset of $\text{int}(\Omega)$, provided that $L^2$ bounds exist.

\begin{theorem}\label{thm:uniform}
For a probability distribution $\nu$ such that $\left\Vert \phi_{\mu}- \phi_{\mu;R} \right\Vert_{L^2(\Omega)} = h(R) \rightarrow 0$ as $R \rightarrow 0$, we have that $\phi_{\mu;R} \rightarrow \phi_{\mu}$ pointwise on any open strict subset of $\Omega$ and uniformly on any closed subset of $\text{int}(\Omega)$.
\end{theorem}
\begin{proof}
First, we prove that $\phi_{\mu;R}$ and $\phi_{\mu}$ can only be infinite on $\partial \Omega$. This follows from the Brenier theorem, namely the fact that $\mu$ is supported on $\text{Dom}(\phi_{\mu})$ and likewise for $\text{Dom}(\phi_{\mu;R})$. Therefore, since $\text{spt}(\mu) = \Omega$, we have that $\phi_{\mu}$ and $\phi_{\mu;R}$ can only be infinite on $\partial \Omega$.

By Proposition~\ref{prop:almosteverywhere}, a fast enough convergent subsequence $\{ \phi_{\mu;R_{k}} \}_{k}$ converges almost everywhere to $\phi_{\mu}$. We now demonstrate that such a subsequence converges pointwise to $\phi_{\mu}$ on compact subsets of $\Omega$. To do this, we now apply Theorem 10.8 from Rockafellar~\cite{rockafellar}, to a strict open convex subset $C$ of $\Omega$. Since $\phi_{\mu;R_{k}} \rightarrow \phi$ in a dense subset $C'$ of $C$, and $\phi_{\mu}$ is finite on $C'$. Therefore, any sequence $\phi_{\mu;R_{k}} \rightarrow \psi$ on $C$, where $\psi$ is convex. The function $\psi$ is equal almost everywhere in $C$ to $\phi_{\mu}$. Therefore, by continuity of convex functions on open sets, $\psi = \phi_{\mu}$ on $C$. The convergence is uniform on closed subsets of $C$ also by Theorem 10.8 of Rockafellar.

What we have just demonstrated is that every sequence $\{ \phi_{\mu;R_{k}} \}_{k}$ that converges quickly enough converges to $\phi_{\mu}$. Now, fix such a compact subset $C$. Suppose that we have a sequence $\{ \phi_{\mu;R_{l}} \}_{l}$, which does not converge to $\phi_{\mu}$ on $C$. We extract a subsequence $\{ \phi_{\mu;R_{l}} \}_{l}$ so that there exists a $y \in C$ such that there exists an $\epsilon>0$ such that $ \vert \phi_{\mu;R_{l}}(y) - \phi(y) \vert > \epsilon$ for every $l$. Then, by the same argument as above, there exists a convergent subsequence of this subsequence, contradicting that we assumed $\phi_{\mu;R_{l}}$ did not converge to $\phi_{\mu}$. Therefore, all sequences $\{ \phi_{\mu;R_{l}} \}_{l}$ converge to $\phi_{\mu}$ on $C$.
\end{proof}

Since $\phi_{\mu;R}$ converges to $\phi_{\mu}$ pointwise, we are able to now prove that $T_{\mu;R}$ converges to $T_{\mu}$ almost everywhere. This follows from the convexity property of the Brenier potentials and due to the fact that the mapping is contained in the subdifferential of the Brenier potential at every point $x$.

\begin{theorem}\label{thm:aemapping}
If $\phi_{\mu;R}$ and $\phi_{\mu}$ satisfy $\phi_{\mu;R} \rightarrow \phi$ pointwise, then $T_{\mu;R} \rightarrow T_{\mu}$ pointwise almost everywhere on $\Omega$, where $T_{\mu}$ is any mapping satisfying $T_{\mu}(x) = \nabla \phi_{\mu}(x)$ for almost every $x \in \Omega$.
\end{theorem}
\begin{proof}
Since $\phi_{\mu;R} \rightarrow \phi_{\mu}$ pointwise on $\text{spt}(\nu)$, then by Rockafellar Theorem 24.5~\cite{rockafellar}, we have that $\lim_{R \rightarrow \infty} \partial \phi_{\mu;R}(x) \subset \phi_{\mu}(x)$ on $\text{int}(\Omega)$, since $\phi_{\mu;R}(x)$ limits to a finite number everywhere on $\text{int}(\Omega)$. By the Brenier-McCann theorem and the Kantorovich Theorem~\cite{maggi}, the optimal mapping satisfies $T_{R} \subset \partial \phi_{\mu;R}$ and $T \subset \partial \phi_{\mu}$. Therefore, for any $x \in F$, where $F \subset \Omega$ denotes the set of differentiability points of $\phi_{\mu}$, we have $\lim_{R \rightarrow \infty} T_{R}(x) \in \partial \phi_{\mu}(x) = \{ \nabla \phi_{\mu}(x) \}$. Since $\phi_{\mu}$ is convex, by Radamacher's theorem, $\phi_{\mu}$ is differentiable almost everywhere and thus $T_{\mu;R}(x) \rightarrow \nabla \phi_{\mu}(x)$ almost everywhere in $\Omega$. Therefore, $T_{\mu;R} \rightarrow T_{\mu}$, where $T_{\mu} = \nabla \phi_{\mu}(x)$ almost everywhere on $\Omega$, since $\partial \Omega$ is a set of measure zero (the boundary of a compact convex set has Lebesgue measure zero).
\end{proof}

\begin{remark}
We cannot expect pointwise or uniform convergence of $T_{\mu;R}$ to $T_{\mu}$, since $T_{\mu}(x) \rightarrow \infty$ as $\vert x \vert \rightarrow \infty$ and also because $T_{\mu}$ is not guaranteed to be continuous.
\end{remark}

\begin{remark}
To recapitulate, what we have shown in Theorems~\ref{thm:uniform} and~\ref{thm:aemapping} is that under the assumptions that 1) $\mu$ and $\nu$ have density functions, 2) the density function of $\mu$ is strictly bounded away from zero, 3) the cost function $c(x,y) = \vert x - y \vert^2$, and 4), $p>n$ and $p>4$, weak convergence of $\nu_{R}$ to $\nu$ leads to pointwise convergence of the Brenier potentials and almost-everywhere convergence for the optimal transport mapping. This is guaranteed merely under the assumption that $\nu_{R}$ converges to $\nu$ weakly.
\end{remark}

\subsection{Explicit Uniform Convergence Rates for Brenier Potentials}\label{sec:explicituniform}

This subsection contains the main results of this manuscript. The results of Section~\ref{sec:pointwise} can be used, in conjunction with H\"{o}lder continuity results on the Brenier potentials, to yield explicit convergence rates for the pointwise convergence of the Brenier potentials. Given the assumptions on the problem in Section~\ref{sec:pointwise}, therefore, leads to convergence rates that are controlled by the rate of convergence of the quantity $W_1(\nu_{R} , \nu)$.

In Delalande and M\`{e}rigot~\cite{merigot1}, the authors cite Berman and Berndtsson~\cite{BermanBerndtsson} as the source of the following result showing the H\"{o}lder continuity property of Brenier potentials on $\Omega$:

\begin{proposition}
If $p>\max\{ n, 2\}$ and $M_{p}<+\infty$ and $0<m<f_{\mu}<M<\infty$ on $\Omega$, then the Brenier potential $\phi$ is H\"{o}lder continuous over $\Omega$. That is, there exists a constant $C = C(n, p, \Omega)>0$ such that:
\begin{equation}\label{eq:holder}
\vert \phi_{\mu}(x) - \phi_{\mu}(x') \vert \leq C(n, p, \Omega) \left( \frac{M_{p}}{m} \right)^{1/p} \Vert x - x' \Vert^{1-\frac{n}{p}},
\end{equation}
for all $x, x' \in \Omega$.
\end{proposition}
\begin{remark}
Note that the result also applies to the Brenier potential $\phi_{\mu;R}$, since the H\"{o}lder constants are the same.
\end{remark}

\subsubsection{Explicit Rates of Convergence for $p > n$ and $p \geq 4$ on $\Omega$}

In this section, for the case $p>n$ and $p \geq 4$, by using the fact that the Brenier potentials are H\"{o}lder continuous, we will be able to derive explicit convergence rates. We will also obtain convergence results in anticipation of future results similar to Theorem~\ref{thm:delalandemerigot}.

Most of the following estimates in this subsection (except those where $\Omega$ is a rectangle) come from first providing an estimate on a ball of radius $\epsilon>0$. Then, this estimate on Euclidean balls can be used on domains that satisfy the inner $\epsilon$-ball property. For the $n$-cube case, we establish the explicit bounds for rectangular regions in Theorem~\ref{thm:rectangular}. We therefore begin with the following technical estimate:

\begin{lemma}\label{lemma:estimate1}
Let $x_0 \in \mathbb{R}^n$ and $\epsilon>0$ and define $B_{\epsilon}(x_0)$. Then, let $\psi(x)$ be a H\"{o}lder continuous function on $\overline{B_{\epsilon}(x_0)}$ with H\"{o}lder exponent $\alpha$ that satisfies $\psi(x) \geq 0$ on $\overline{B_{\epsilon}(x_0)}$ and there exists $\tilde{x} \in \overline{B_{\epsilon}(x_0)}$ such that $\psi(\tilde{x}) = 0$. Also, let $\psi(x)$ satisfy $\Vert \psi \Vert_{L^1(B_{\epsilon}(x_0))} \leq h(R)$, where $h(R) \rightarrow 0$ as $R \rightarrow \infty$. Then, there exists an $R_0>0$ and a constant $C = C(n, \alpha, C_{H})>0$ such that $\max_{x \in \overline{B_{\epsilon}(x_0)}} \vert \psi(x) \vert \leq C(n, \alpha, C_{H}) (h(R))^{\frac{\alpha}{\alpha + n}}$ for all $R \geq R_0$, where $C_{H}$ is the H\"{o}lder constant of $\psi(x)$.
\end{lemma}

\begin{proof}
We construct a function $\tilde{\psi}$ satisfying $\tilde{\psi}(\tilde{x}) = 0$ for some $\tilde{x} \in B_{\epsilon}(x_0)$ and $\tilde{\psi}(x) \geq 0$ on $B_{\epsilon}(x_0)$, but the maximum of $\tilde{\psi}$ on $B_{\epsilon}(x_0)$ we take to be at some $x^{*} \in \partial B_{\epsilon}(x_0)$ and $\tilde{\psi}(x)$ will increase as quickly as possible (controlled by the H\"{o}lder condition with constant $C_{H}$ and exponent $\alpha$) from $\tilde{x}$ to $x^{*}$. Let $\Delta x := \Vert \tilde{x} - x^{*} \Vert$. The function $\tilde{\psi}(x)$ will be thus chosen explicitly as:
\begin{equation}
\tilde{\psi}(x) :=
\begin{cases}
C_{H}(\Delta x - \Vert x - x^{*} \Vert)^{\alpha}, &\text{for} \ \Vert x - x^{*} \Vert \leq \Delta x \\
0, &\text{otherwise}.
\end{cases}
\end{equation}
By this construction, any $\psi$ satisfying the hypotheses of the lemma such that $\int_{B_{\epsilon}(x_0)} \psi(x)dx \leq \int_{B_{\epsilon}(x_0)} \tilde{\psi}(x)dx$ will then automatically satisfy $\max_{x \in \overline{B_{\epsilon}(x_0)}} \psi(x) \leq \max_{x \in \overline{B_{\epsilon}(x_0)}} \tilde{\psi}(x)$.

We now compute the integral of $\tilde{\psi}$:
\begin{align}
\int_{B_{\epsilon}(x_0)} \tilde{\psi}(x)dx &= C_{H} \int_{B_{\epsilon}(x_0) \cap B_{\Delta} x^{*}} (\Delta x - \Vert x - x^{*} \Vert)^{\alpha}dx \\
&= C\vert \mathbb{S}^{n-1} \vert \frac{\vert B_{\epsilon}(x_0) \cap B_{\Delta x}(x^{*}) \vert }{\vert B_{\Delta x}(x^{*}) \vert} \int_{0}^{\Delta x} r^{n-1} (\Delta x - r)^{\alpha}dr.
\end{align}
For $\Delta x$ small enough, we have:
\begin{equation}\label{eq:deltax}
\frac{1}{4}\vert B_{\Delta x}(x^{*}) \vert \leq \vert B_{\epsilon}(x_0) \cap B_{\Delta x}(x^{*}) \vert.
\end{equation}
This comes from the observation that in the limit, the intersection of $B_{\epsilon}(x_0)$ and $B_{\Delta x}(x^{*})$ is half of $B_{\Delta x}(x^{*})$. That is, there exists a value $0<\overline{\Delta x} \leq 2 \epsilon$ such that for all $\Delta x \leq \overline{\Delta x}$ we have Inequality~\eqref{eq:deltax}. We compute the following integral exactly via repeated integration by parts:
\begin{equation}
\int_{0}^{\Delta x} r^{n-1}(\Delta x - r)^{\alpha}dr = \frac{(n-1)!}{\prod_{i=1}^{n}(\alpha + i)}\Delta x^{\alpha + n}.
\end{equation}
We overestimate the size of $\Delta x$ by requiring:
\begin{align}
\int_{B_{\epsilon}(x_0)} \psi(x)dx &\geq \frac{1}{4} C_{H} \vert \mathbb{S}^{n-1} \vert \frac{(n-1)!}{\prod_{i=1}^{n} (\alpha + i)} \Delta x^{n + \alpha} \\
&= h(R).
\end{align}
Hence we choose:
\begin{equation}
\Delta x = \left( \frac{h(R)}{\beta} \right)^{\frac{1}{n+ \alpha}},
\end{equation}
where $\beta = \frac{1}{4} C_{H} \vert \mathbb{S}^{n-1} \vert \frac{(n-1)!}{\prod_{i=1}^{n} (\alpha + i)}$. Let $R_0>0$ denote the value such that for all $R \geq R_0$, we have $\Delta x \leq \overline{\Delta x}$. Then, this assures that:
\begin{align}
\int_{B_{\epsilon}(x_0)} \tilde{\psi}(x)dx &\geq \frac{C_{H}}{4} \vert \mathbb{S}^{n-1} \vert \frac{(n-1)!}{\prod_{i=1}^{n} (\alpha + i)} \Delta x^{n + \alpha} \\
&= h(R) \\
&\geq \int_{B_{\epsilon}(x_0)} \psi(x)dx,
\end{align}
for all $R \geq R_0$. Thus, for all $R \geq R_0$, we have the bound:
\begin{equation}
\max_{x \in \overline{B_{\epsilon}(x_0)}}\psi(x) \leq \max_{x \in \overline{B_{\epsilon}(x_0)}} \tilde{\psi}(x) \leq C_{H} \Delta x^{\alpha} = \frac{C_{H}}{\beta^{\frac{\alpha}{\alpha+n}}} (h(R))^{\frac{\alpha}{\alpha + n}},
\end{equation}
for all $R>0$.
\end{proof}

Now follows what we consider to be the main result of this manuscript. Using Lemma~\ref{lemma:estimate1}, then we are able to prove the following:
\begin{theorem}\label{thm:explicitrates1}
Let $\Omega$ satisfy the inner $\epsilon$-ball property, see Definition~\ref{def:epsilonball}, for some $\epsilon>0$. Let $0<m \leq f_{\mu} \leq M < \infty$ and let $p>n$, $p \geq 4$ and let $\Vert \phi_{\mu;R} - \phi_{\mu} \Vert_{L^1(\Omega)} \leq h(R)$. Then, there exists a constant $C = C(n, p, \Omega, m, M, C_{H})$ and a value $R_0 >0$ such that for all $R \geq R_0$, we have $\max_{x \in \Omega} \vert \phi_{\mu;R}(x) - \phi_{\mu}(x) \vert \leq C(n, p, \Omega, m, M, C_{H}) (h(R))^{\frac{1 - \frac{n}{p}}{1+\frac{n}{q}}}$, where $q$ satisfies $\frac{1}{p}+\frac{1}{q} = 1$.
\end{theorem}

\begin{proof}
Define $\psi(x) := \vert \phi_{\mu;R}(x) - \phi_{\mu}(x) \vert$. Then, since $\int_{\Omega} \phi_{\mu;R}(x)d\mu(x) = \int_{\Omega} \phi_{\mu}(x) d\mu(x) = 0$, the density function $f_{\mu}$ is bounded away from zero and both $\phi_{\mu}$ and $\phi_{\mu;R}$ are continuous on $\overline{\Omega}$, we have that $\psi(x)$ must equal zero somewhere in $\Omega$. Also, $\psi(x) \geq 0$. Furthermore, since both $\phi_{\mu;R}$ and $\phi_{\mu}(x)$ are H\"{o}lder continuous with the same constants and exponent $\alpha = 1 - \frac{n}{p}$, then $\psi(x)$ is H\"{o}lder continuous (with twice the constant) with the same exponent $\alpha = 1 - \frac{n}{p}$. Since we have $L^2$ rates and we are integrating over a compact set $\Omega$, this immediately leads to $L^1$ convergence rates over $\Omega$. Thus, we have that $\int_{\Omega} \psi(x)dx = h(R)$ and $h(R) \rightarrow 0$ as $R \rightarrow \infty$.

We can bound the maximum of $\psi(x)$ by finding a function $\tilde{\psi}(x)$ with the same properties, but that takes the volume $h(R)$ in the least area. This is done by fixing the maximum of $\tilde{\psi}(x)$ to be attained at a point $x^{*} \in \partial \Omega$ and having the function decrease as quickly as possible to zero and be zero elsewhere. That is, we desire $\tilde{x}$ to be a point at which $\tilde{\psi}(\tilde{x}) = 0$. Let $C_{H}$ be the H\"{o}lder constant of $\psi$. Then,
\begin{equation}\label{eq:tildepsi}
\tilde{\psi}(x) := \begin{cases}
C_{H}(\Vert x^{*} - \tilde{x} \Vert - \Vert x - x^{*} \Vert)^{\alpha}, &\text{for} \ \Vert x - x^{*} \Vert \leq \Vert x^{*} - \tilde{x} \Vert \\
0, &\text{otherwise}.
\end{cases}
\end{equation}
Fix $R>0$. We then choose $\tilde{\psi}(x)$ by choosing $\tilde{x}$ to be a point which satisfies:
\begin{equation}
\int_{\Omega} \tilde{\psi}(x)dx = h(R).
\end{equation}
Now, we define another function $\Theta(x)$ in the following way. Since $\Omega$ satisfies the inner $\epsilon$-ball property, there exists a point $x_0 \in \Omega$ such that $x^{*} \in \partial B_{\epsilon}(x_0)$. Then, we define $\Theta$ as:
\begin{equation}
\Theta(x) := \begin{cases}
C_{H}(\Vert x^{*} - y \Vert - \Vert x - x^{*} \Vert)^{\alpha}, &\text{for} \  x \in \{\Vert x - x^{*} \Vert \leq \Vert x^{*} - y \Vert\} \cap \{ B_{\epsilon}(x_0) \} \\
0, &\text{otherwise}.
\end{cases}
\end{equation}
And settle the choice of $y$ to be a point which satisfies:
\begin{equation}
\int_{\Omega} \Theta(x)dx = h(R).
\end{equation}
Notice that there exists an $R_0>0$ such that for all $R \geq R_0$, $\Theta(x)$ and $\tilde{\psi}(x)$ take the value zero somewhere in $B_{\epsilon}(x_0)$. Also, notice by construction that $\max_{x \in \Omega} \tilde{\psi}(x) \leq \max_{x \in \Omega} \Theta(x)$. We now apply Lemma~\ref{lemma:estimate1} with the H\"{o}lder constants and exponents of $\psi$ to bound $\Theta$ and get:
\begin{equation}
\max_{x \in \Omega} \psi(x) \leq \max_{x \in \Omega} \tilde{\psi}(x) \leq \max_{x \in \Omega} \Theta(x) \leq C(n, p, \Omega, m, M, C_H) (h(R))^{\frac{1-\frac{n}{p}}{1+\frac{n}{q}}}.
\end{equation}
\end{proof}

\begin{remark}
We emphasize here that the rates of convergence given above are not asymptotic, but rather explicit.
\end{remark}

For closed rectangular domains, i.e. $\Omega = \{ x \in \mathbb{R}^n : a \leq x \leq b, a, b \in \mathbb{R}^n, a_i < b_i \ \text{for all} \ i \}$, which is what we will be using in computation, see Section~\ref{sec:computation}, we can produce a similar estimate.
\begin{theorem}\label{thm:rectangular}
For rectangular domains $\Omega = \{ x \in \mathbb{R}^n : a \leq x \leq b, a, b \in \mathbb{R}^n, a_i < b_i \ \text{for all} \ i \}$, let $p > n$, $p \geq 4$ and $\Vert \phi_{R} - \phi \Vert_{L^{1}(\Omega)} \leq h(R)$. Then, we have that:
\begin{equation}
\max_{x \in \Omega} \vert \phi_{\mu;R}(x) - \phi_{\mu}(x) \vert \leq C(n, p, \Omega, m, M, C_{H}) (h(R))^{\frac{1- \frac{n}{p}}{1 + \frac{n}{q}}}.
\end{equation}
\end{theorem}
\begin{proof}
As in the proof of Theorem~\ref{thm:explicitrates1}, we have that $\int_{\Omega} \psi(x)dx = h(R)$, where $\Omega$ is a rectangular domain that, notably, does not satisfy the $\epsilon$-ball property. As in the same proof, we find a function $\tilde{\psi}$ of the form in Equation~\eqref{eq:tildepsi} such that the maximum of the function occurs at $x^{*} \in \partial \Omega$, which is a corner of the rectangular domain $\Omega$, i.e. a point where $x^{*}_i = a_i$ or $x^{*}_i = b_i$ for all $i$ and we choose $\tilde{x}$ such that $\int_{\Omega} \tilde{\psi}(x)dx = h(R)$. We choose $\Delta x = \Vert x^{*} - \tilde{x} \Vert$ such that:
\begin{align}
h(R) &= \int_{\Omega} \tilde{\psi}(x)dx \\
&= \frac{1}{2^{n}} \int_{B_{\Delta x} (x^{*})} C_{H} (\Delta x - \Vert x - x^{*} \Vert)^{\alpha}dx \\
&= \frac{C_{H}}{2^{n}} \vert \mathbb{S}^{n-1} \vert \int_{0}^{\Delta x} r^{n-1} (\Delta x - r)^{\alpha}dx,
\end{align}
where the second equality is true for $\Delta x$ small enough, i.e. $\Delta x \leq \overline{\Delta x} < \min_{i} b_i - a_i$. Using the computation of $\int_{0}^{\Delta x} r^{n-1} (\Delta x - r)^{\alpha}dx$ from the proof of Lemma~\ref{lemma:estimate1}, we get that choosing $\Delta x = \min \left\{ \overline{\Delta x}, \left( \frac{h(R)}{\tilde{\beta}} \right)^{\frac{1}{\alpha + n}} \right\}$, where $\tilde{\beta} = 2^{-n} \vert \mathbb{S}^{n-1} \vert C_{H} \frac{(n-1)!}{\prod_{i=1}^{n} (\alpha + i)}$, we have that:
\begin{align}
\max_{x \in \Omega} \psi(x) &\leq \max_{x \in \Omega} \tilde{\psi}(x) \\
&\leq C(n, p, \Omega, m, M, C_H) (h(R))^{\frac{\alpha}{\alpha + n}} \\
&=C(n, p, \Omega, m, M, C_H) (h(R))^{\frac{1 - \frac{n}{p}}{1 + \frac{n}{q}}}.
\end{align}
\end{proof}

\begin{remark}
We emphasize that the estimates for rectangular domains are not asymptotic, but rather are explicit.
\end{remark}

\begin{remark}\label{rmk:geq2rates}
Although it can be shown that the potential functions $\phi_{\mu}$ and $\phi_{\mu;R}$ are equi-Lipschitz in, say, a subset $\Omega ' \subset \text{int}(\Omega)$ (that satisfies an $\epsilon$-ball property with a different value of $\epsilon$ than for $\Omega$), it does not seem possible to generally obtain accelerated convergence rates due to this equi-Lipschitz property in $\Omega '$, even though we expect that rates could be established on such domains $\Omega '$ due to the equi-Lipschitz property. This is due to the fact that in order to apply the techniques of Lemma~\ref{lemma:estimate1}, Theorem~\ref{thm:explicitrates1} or Theorem~\ref{thm:rectangular}, we must be guaranteed that $\phi_{\mu}(x) - \phi_{\mu;R}(x)$ attains a zero in $\Omega '$ uniformly in $R$ (one can easily show that for a fixed $R$ there must be a zero in the interior of $\Omega$). It is not at this point obvious to the author how to overcome this difficulty. If the zeros of the function $\phi_{\mu}(x) - \phi_{\mu;R}(x)$ all remained in $\Omega '$, then one would immediately have the asymptotic rate $\max_{x \in \Omega_{\delta}} \vert \phi_{\mu}(x) - \phi_{\mu;R}(x) \vert \leq C(n, p, \Omega, m, M, \Omega ') (h(R))^{\frac{1}{1+n}}$.
\end{remark}

\subsection{Convergence and Other Properties of the Inverse Brenier Potential}\label{sec:inversebrenier}

In this subsection, we will prove some convergence results for the inverse Brenier potential. Recall that for $p >n$ and $p \geq 4$ the Brenier potential is H\"{o}lder continuous and for $p \neq n$ it is not necessarily H\"{o}lder continuous and may even be unbounded. We show that the inverse Brenier potential, by contrast, is Lipschitz. We also derive the convergence of the inverse Brenier potential in the $p > n$ and $p \leq n$, in the latter case by adding the additional assumption that the density function of $\nu$ is bounded away from zero.

\begin{lemma}\label{lemma:lipschitz}
The inverse Brenier potential $\phi_{\nu}$ and $\phi_{\nu;R}$ are Lipschitz with constant $\tilde{R}$, where $\tilde{R} = \sup_{x \in \Omega} \Vert x \Vert$.
\end{lemma}
\begin{proof}
Let $\tilde{R} = \sup_{x \in \Omega} \Vert x \Vert$. The result follows from the Brenier-McCann theorem (see Theorem 6.1 in Maggi~\cite{maggi}), the fact that mass is transported to a target measure with compact support, and by the continuity of subdifferentials of convex functions. We show the result for $\phi_{\nu}$, since the argument for $\phi_{\nu;R}$ is identical, \textit{mutatis mutandis}. Fix $x, y \in \mathbb{R}^d$. Since $\phi_{\nu}$ is a Brenier potential, by the Brenier-McCann theorem, see Theorem 6.1 in Maggi~\cite{maggi}, we have that since the support of the target set is compact, that $\phi_{\nu}$ can be chosen such that $\phi_{\nu}(x) \neq 0$ for all $x \in \mathbb{R}^d$. By the fact that $\phi_{\nu}$ is convex, we then have that $\nabla \phi_{\nu}(x)$ exists for $x \in F$, where $F$ is a dense subset of $\mathbb{R}^d$ by Theorem 25.5 of Rockafellar~\cite{rockafellar}. Then, by the Brenier-McCann theorem, $\nabla \phi_{\nu}(x)$ satisfies $\vert \nabla \phi_{\nu} \vert \leq \tilde{R}$ in $F$, since $\nabla \phi_{\nu}(x)$ is contained in the support of the target measure for every $x \in F$. By the continuity of subdifferentials, for every $\epsilon>0$, there exists a $\delta>0$ such that $\partial \phi_{\nu}(B_{\delta}(x)) \subset B_{\epsilon} (\nabla \phi_{\nu}(x))$. Fix $x, y \in \mathbb{R}^d$ and $\epsilon>0$. Since $\phi_{\nu}$ is differentiable on $F$ which is dense in $\mathbb{R}^d$, we can find a point $\tilde{x}$ that satisfies $\vert x - \tilde{x} \vert \leq \delta$, where $\tilde{x}$ is a point of differentiability such that every $z \in \partial \phi_{\nu}(x)$ satisfies $\vert z \vert \leq \tilde{R}+\epsilon$. Using this, and the definition of the subdifferential, we get:
\begin{align}\label{eq:subdif1}
\phi_{\nu}(x)-\phi_{\nu}(y) &\leq z \cdot (x - y).
\end{align}
Likewise, we can guarantee that every $z' \in \partial \phi_{\nu}(y)$ satisfies $\vert z' \vert \leq \tilde{R}+\epsilon$. Using this and Inequality~\eqref{eq:subdif1}, we get:
\begin{equation}
-z' \cdot (x - y) \leq \phi_{\nu}(x)-\phi_{\nu}(y) \leq z \cdot (x - y).
\end{equation}
Therefore,
\begin{align}
\vert \phi_{\nu}(x)-\phi_{\nu}(y) \vert &\leq \max \{ \vert z \vert, \vert z' \vert \} \vert x - y \vert, \\
&\leq (\tilde{R} + \epsilon) \vert x - y \vert.
\end{align}
Taking $\epsilon \rightarrow 0$ completes the proof.
\end{proof}




For the case $p>n$ and $p \geq 4$, once we have convergence rates for the Brenier potential on $\Omega$, we immediately get the same convergence rates for the inverse Brenier potential.

\begin{theorem}\label{thm:pgn}
For $p>n$, assume that $\phi_{\mu}$ and $\phi_{\mu;R}$ satisfy $\vert \phi_{\mu}(x) - \phi_{\mu;R}(x) \vert \leq k(R)$ for all $x \in \Omega$. Then, $\phi_{\nu}$ and $\phi_{\nu;R}$ satisfy $\vert \phi_{\nu}(y) - \phi_{\nu;R}(y) \vert \leq k(R)$ for $y \in \text{spt}(\nu)$.
\end{theorem}

\begin{proof}
By definition, the inverse Brenier potential satisfies:
\begin{equation}\label{eq:fenchel}
\phi_{\nu}(y) = \sup_{x} \{ x \cdot y - \phi_{\mu}(x) \}.
\end{equation}

First, we show that for every $y \in \text{spt}(\nu)$, we have the existence of a point $\tilde{x} \in \text{spt}(\mu)$ such that $\phi_{\nu}(y) = \tilde{x} \cdot y - \phi_{\mu}(\tilde{x})$. We note that equality in the Fenchel-Legendre dual formula~\eqref{eq:fenchel} is attained if and only if $x \in \text{Dom}(\phi_{\mu})$ and $y \in \partial \phi_{\mu}(x)$. However, we know that by the Brenier-McCann theorem, we have $\text{Cl} (\nabla \phi_{\mu} (F \cap \text{spt}(\mu)))  = \text{spt}(\nu)$. Thus, fix $y \in \text{spt}(\nu)$. We then have that there exists a sequence $(x_k, y_k)$ such that $y_k = \nabla \phi_{\mu}(x_k)$, $y_k \rightarrow y$ and $x_k \in \text{spt}(\mu)$. Using this sequence, we get:
\begin{equation}\label{eq:maximum}
\phi_{\nu}(y_k) = x_k \cdot y_k - \phi_{\mu}(x_k).
\end{equation}
Since $x_k \in \text{spt}(\mu) \subset \Omega$, we have that there exists a subsequence $\{ x_{k_{l}} \}$ such that $x_{k_{l}} \rightarrow \tilde{x}$, where $\tilde{x} \in \Omega$. Also, by the H\"{o}lder continuity of $\phi_{\mu}$ in $\Omega$, we have that $\phi_{\mu}(x_{k_{l}}) \rightarrow \phi_{\mu}(\tilde{x})$. Since $\phi_{\nu}$ is continuous on all of $\mathbb{R}^d$, since it is Lipschitz continuous by Lemma~\ref{lemma:lipschitz}, we have that $\phi_{\nu}(y_{k_{l}}) \rightarrow \phi_{\nu}(y)$. Hence, we take the limit of Equation~\eqref{eq:maximum} and get:
\begin{equation}
\phi_{\nu}(y) = \tilde{x} \cdot y - \phi_{\mu}(\tilde{x}).
\end{equation}
This shows the existence of $\tilde{x} \in \Omega$ such that Equation~\eqref{eq:fenchel} has a supremum. Using the point $\tilde{x}$, we obtain:
\begin{align}
\phi_{\nu}(y) &= \tilde{x} \cdot y - \phi_{\mu}(\tilde{x}) \\
\phi_{\nu;R}(y) &\geq \tilde{x} \cdot y - \phi_{\mu;R}(\tilde{x}).
\end{align}
Hence,
\begin{equation}
\phi_{\nu;R}(y) - \phi_{\nu}(y) \geq \phi_{\mu}(\tilde{x}) - \phi_{\mu;R}(\tilde{x}) \geq -k(R).
\end{equation}
Let $\tilde{x}_{R}$ designate any point that satisfies $\phi_{\nu;R}(y) = \tilde{x}_{R} \cdot y - \phi_{\mu;R}(\tilde{x}_{R})$. The existence of such a point, again, can be made by the above reasoning applied to the function $\phi_{\mu;R}$. Using this, point, we get:
\begin{equation}
\phi_{\nu;R}(y) - \phi_{\nu}(y) \leq  \phi_{\mu;R}(\tilde{x}_R) - \phi_{\mu}(\tilde{x}_R),
\end{equation}
applying the uniform bounds $\vert \phi_{\mu;R}(z) - \phi_{\mu}(z) \vert \leq k(R)$ for any $z$ then allows us to conclude:
\begin{equation}
\vert \phi_{\nu;R}(y) - \phi_{\nu}(y) \vert \leq  k(R).
\end{equation}
\end{proof}

For the case where $p \leq n$ or $p <4$, we do not have explicit rates, but are assured pointwise convergence on $\Omega$ by Theorem~\ref{thm:uniform} once $L^2$ rates are established. Before doing this, we start with a brief geometrical lemma, which shows that non-empty, convex domains $\Omega$ that satisfy the $\epsilon$-ball property have a boundary that is a Riemannian manifold, which allows us to assert the existence of a normal vector at each boundary point of $\Omega$.

\begin{lemma}
If $\Omega$ is non-empty, connected, convex, and satisfies the $\epsilon$-ball property, then $\partial \Omega$ is a smooth manifold.
\end{lemma}
\begin{proof}
We first note that if $\Omega$ satisfies the $\epsilon$-ball property, it has a non-empty, connected interior, since it is convex. Also, $\partial \Omega$ is a closed set, since $\Omega$ is closed. The set $\partial \Omega$ is of Hausdorff dimension $\mathbb{R}^{n-1}$.

For each $x \in \partial \Omega$, we have a point $x_0 \in \Omega$ such that $x \in \overline{B_{\epsilon}(x_0)}$. By the supporting hyperplane theorem, then, we see that a supporting hyperplane at $x$ for $\Omega$ is a supporting hyperplane at $x$ of $\overline{B_{\epsilon}(x_0)}$. Since $\overline{B_{\epsilon}(x_0)}$ is a closed ball, this is the unique supporting hyperplane and therefore $\partial \Omega$ has unique tangent planes at each point.

Moreover, let $x_{n} \rightarrow x$, where $x_{n} \in \partial \Omega$ and $x \in \Omega$ and let $B_{\epsilon}(z_{n})$ be a sequence of open Euclidean balls such that $x_{n} \in \overline{B_{\epsilon}(z_n)}$. Then, let $\mathfrak{T}_{x_{n}}$ be the sequence of uniquely defined tangent planes of $\overline{B_{\epsilon}(z_{n})}$ at $x_{n}$. Upon possible extraction of a sequence, since $\mathfrak{T}_{x_{n}}$ are defined by the sequence $x_{n}$, which is convergent, and normal directions of the tangent plane $\hat{n}_{\mathfrak{T}_{x_{n}}} \in \mathbb{S}^{n-1}$, we have that $\mathfrak{T}_{x_{n}} \rightarrow \mathfrak{T}_{x}$, where $\mathfrak{T}(x) \subset \mathbb{R}^n$ is an $n-1$ hyperplane passing through $x$. By the supporting hyperplane theorem, we must have that $\mathfrak{T}(x)$ equals the supporting hyperplane of $\partial \Omega$ passing through $x$. Thus, every sequence of tangent planes must converge to the same tangent plane. Thus, we have continuity of tangent planes. The supporting hyperplane theorem rules out the possibility of self-intersections.

Therefore, $\partial \Omega$ is a smooth manifold.
\end{proof}

In the context of our problem, if we assume that $\nu$ ``has mass at infinity in every direction" and the normal vector $\hat{n}_{x}$ at a point $x \in \partial \Omega$ are unique, then the Brenier potentials have a ``stadium-like effect", which means that the gradients of Brenier potentials will will point in the direction of the normal as one approaches the boundary and the gradients blow up if and only if the boundary is approached.

\begin{lemma}\label{lemma:gradientexplosiononboundary}
Let $\Omega$ be compact and convex and have nonempty interior. For points of differentiability $x_k \in \Omega$, we have that $\Vert \nabla \phi_{\mu}(x_k) \Vert \rightarrow + \infty$ only if $\text{dist}(x_k, \partial \Omega) \rightarrow 0$.
\end{lemma}
\begin{proof}
From the Brenier theorem, see Theorem 4.2 in Maggi~\cite{maggi}, we get that $\phi_{\mu}(x) < \infty$ $\mu$-a.e. Since $\mu$ has a density function bounded away from zero in $\Omega$, and $\partial \Omega$ has measure zero, then $\phi_{\mu}(x)$ can only be infinite, if it is, on the boundary $\partial \Omega$. This is due to the fact that essential domain of $\phi_{\mu}$ is convex set, but that $\Omega \setminus \text{dom}(\phi_{\mu})$ has measure zero. Since removing $\partial \Omega$ only removes a set of measure zero, leaving $\text{int}(\Omega)$, which is still convex, then clearly $\phi_{\mu}$ might be infinite on the boundary. Removing any point from $\text{int}(\Omega)$ leaves a set that is no longer convex. So, we see that it is only on $\partial \Omega$ that $\phi_{\mu}$ may take infinite values and consequently, $\phi_{\mu}$ is finite for all $x \in \text{int}(\Omega)$. Furthermore, $\phi_{\mu}$ cannot equal $-\infty$ anywhere. If it did, then $\phi_{\mu}$ could only have finite values on the boundary of the set where $\phi_{\mu}$ equalled $-\infty$, which is a set of measure zero. Since $\phi_{\mu}$ is finite in $\text{int}(\Omega) \neq \emptyset$, we see that $\phi_{\mu}$ cannot equal $-\infty$ anywhere.


By the Brenier theorem, since the support of $\nu$ is unbounded, and the fact that $\phi_{\mu}$ is differentiable almost everywhere in $\text{int}(\Omega)$, we have that there exists a sequence $\{ z_k \}$ of points $z_k \in \text{int}(\Omega)$ such that $\Vert \nabla \phi_{\mu}(z_k) \Vert \rightarrow \infty$. Upon extraction of a subsequence, we may claim that there exists such a sequence satisfying $z_k \rightarrow z \in \Omega$. Now, we show that such a point $z$ can only be on the boundary. This is proved from the observation that once a one-dimensional convex function attains a ``slope" of $+\infty$, the function must be equal to $+\infty$ ``beyond" that point, which will be used to derive a contradiction in the interior of the essential domain where $\phi_{\mu}$ is finite.

Let $\{z_k \}$ be a sequence such that $z_k \rightarrow z' \in \Omega$ and $\Vert \nabla \phi_{\mu}(z_k) \Vert \rightarrow \infty$. For each $k$, there exists a direction vector $\hat{z}_k \in \mathbb{S}^{d-1}$ such that $\nabla \phi_{\mu}(z_k) = \hat{z}_k \vert \nabla \phi_{\mu}(z_k) \vert$. Since $\mathbb{S}^{d-1}$ is a compact set, upon further extraction of a subsequence, we have that $\hat{z}_{k} \rightarrow \hat{z} \in \mathbb{S}^{n-1}$. By the continuity of subdifferentials, $+ \infty \in \partial \phi_{\mu}(z')$. Since $\Omega$ has nonempty interior, there exists a point $z_0 \in \Omega$ such that $(z_0 - z')\cdot \hat{z} \neq 0$.

Suppose that $z' \in \text{int}(\Omega)$. Define the line from $z_0$ to $z'$ via $l(t) = ((\vert z' - z_0 \vert -t)z_0 + tz') /\vert z' - z_0 \vert$ for $t \in [0,\vert x' - z_0 \vert]$. Since we assumed that $z' \in \text{int}(\Omega)$, we may extend $l(t)$ to the boundary. Call the point on the boundary $z_b$. Thus, we define a new line $L(t) := ((\vert z_0 - z_b \vert -t)z_0 + tz') /\vert z' - z_b \vert$. Then, define the function $\psi(t) := \phi_{\mu}(L(t))$. The function $\psi(t)$ is a convex function on $[0, \vert z_b - z_0 \vert]$. Since $+ \infty \in \partial \phi_{\mu}(z')$, and since $(z_0 - z')\cdot \hat{z} \neq 0$ we have $+ \infty \in \partial \psi(t_0)$. Hence, for $t >t_0$, $\psi(t)  = +\infty$. Since there exists a value $t_m$ that satisfies $t_0 < t_m < \vert z_b - z_0 \vert$ such that $\psi(t_m) = +\infty$ and $L(t_m) \in \text{int}(\Omega)$, we have found a point in $\text{int}(\Omega)$ that equals $+\infty$, in contradiction to the fact that infinite values can only be obtained on the boundary. Therefore, $z' \in \partial \Omega$, which means that a gradient whose magnitude is equal to $+\infty$ can only exist on the boundary.
\end{proof}

\begin{lemma}\label{lemma:gradientexplosion}
Let $\Omega$, compact and convex and have a nonempty interior and let the Gauss map be injective (i.e. $\hat{n}_{x} \neq \hat{n}_{y}$ for $x \neq y \in \partial \Omega$). Assume that there exists a compact set $K \subset \mathbb{R}^n$ such that $\text{spt}(\nu) = \tilde{K} \cup (\mathbb{R}^n \setminus K)$, where $\tilde{K} \subset K$. For points of differentiability $x_k \in \Omega$, we have $\Vert \nabla \phi_{\mu}(x_k) \Vert \rightarrow + \infty$ iff $\text{dist}(x_k, \partial \Omega) \rightarrow 0$ and $\nabla \phi_{\mu}(x_k) /\vert \nabla \phi_{\mu}(x_K) \vert \rightarrow \hat{n}_{x}$ if $x_k \rightarrow x \in \partial \Omega$.
\end{lemma}
\begin{proof}
From Lemma~\ref{lemma:gradientexplosiononboundary}, we have that the gradients diverge, if they do, on $\partial \Omega$.

By our assumptions on the support of $\nu$ and the Brenier theorem, the mapping must take mass to every direction at infinity. That is, for every direction $\hat{z} \in \mathbb{S}^{d-1}$, we have that there exists a sequence $\{ x_k \}$ such that $\Vert \nabla \phi_{\mu}(x_k) \Vert \rightarrow \infty$ and $\nabla \phi_{\mu}(x_k) / \Vert \nabla \phi_{\mu}(x_k) \Vert \rightarrow \hat{z}$. Let $\{ x_k \}_{k}$ be any such sequence. We will now show that if $x_k \rightarrow x' \in \partial \Omega$, then $\nabla \phi_{\mu}(x_k) / \Vert \nabla \phi_{\mu}(x_k) \Vert \rightarrow \hat{n}_{x}$. We show that mass at $x'$ cannot be taken to infinity in any other direction. Suppose that $\hat{z} \neq \hat{n}_{x}$ and $x_k \rightarrow x$ but $\vert \nabla \phi_{\mu}(z_k) \vert \rightarrow \infty$ but $\nabla \phi_{\mu}(x_k) / \Vert \nabla \phi_{\mu}(x_k) \Vert \rightarrow \hat{z}$. In this case, since $\phi_{\mu}$ lies above its supporting hyperplanes, but as we approach the boundary, we would have a vertical hyperplane that is not tangent to $\partial \Omega$. More precisely, fix $x \in \text{int}(\Omega)$ such that there exists a constant $m>0$ and a value $K > 0$ such that $\hat{z} \cdot (x - x_k) \geq m>0$ for all $k \geq K$. This is possible, since $\hat{z} \neq \hat{n}_{x'}$ and thus there exist points in $\text{int}(\Omega)$ that are on either side of the half-space defined by the hyperplane with normal $\hat{z}$ passing through the point $x'$. Thus, since $\phi_{\mu}$ is convex, it satisfies:
\begin{equation}
\phi_{\mu}(x) \geq \phi_{\mu}(x_k) + \nabla \phi_{\mu}(x_k) \cdot (x - x_k).
\end{equation}
By the fact that the direction of the gradient converges to $\hat{z}$, we have that there exists a $K_2 >0$ such that for all $k \geq K_2$, we have:
\begin{equation}
\phi_{\mu}(x) \geq \phi_{\mu}(x_k) + \frac{m}{2} \vert \nabla \phi_{\mu}(x_k) \vert.
\end{equation}
Since $\Omega$ is compact, $\phi_{\mu}(x_k) \geq M \in \mathbb{R}$. Therefore we have:
\begin{equation}
\phi_{\mu}(x) \geq M + \frac{m}{2} \vert \nabla \phi_{\mu}(x_k) \vert.
\end{equation}
Therefore, as $k \rightarrow \infty$, we get that $\phi_{\mu}(x) = + \infty$. Since $x \in \text{int}(\Omega)$, we have a contradiction. Thus, we have showed that any sequence of points $\{ x_k \}$ such that $x_k \rightarrow x' \in \partial \Omega$ and the gradients at $x_k$ blow up as $k \rightarrow \infty$ must converge to the normal direction. Otherwise, we could take a subsequence that did not converge to $\hat{x'}$ and derive a contradiction.

Since the normal vectors are unique there is a one-to-one correspondence between $x \in \partial \Omega$ and $\hat{n} \in \mathbb{S}^{n-1}$. Therefore, any sequence $\{ x_{k} \}$ such that $x_k \rightarrow x \in \partial \Omega$ satisfies $\nabla \phi_{\mu}(x_k) / \vert \nabla \phi_{\mu}(x_k) \vert \rightarrow \hat{n}_{x}$ and $\vert \phi_{\mu}(x_k) \vert \rightarrow \infty$.
\end{proof}

\begin{theorem}\label{thm:pleqninverse}
For $p\leq n$ or $p<4$, assume that $\phi_{\mu;R}$ and $\phi_{\mu}$ satisfy $\phi_{\mu;R}(x) \rightarrow \phi_{\mu}(x)$ for any $x \in \text{int}(\Omega)$. Assume that there exists a compact set $K \subset \mathbb{R}^n$ such that $\text{spt}(\nu) = \tilde{K} \cup (\mathbb{R}^n \setminus K)$, where $\tilde{K} \subset K$ and that the Gauss map of $\partial \Omega$ is injective. Then, $\phi_{\nu}$ and $\phi_{\nu;R}$ satisfy $\phi_{\nu;R}(y) \rightarrow \phi_{\nu}(y) \ \forall y \in \text{spt}(\nu)$.
\end{theorem}

\begin{proof}
The notation in this proof follows some notation from Theorem~\ref{thm:pgn}. For a fixed $y \in \text{spt}(\nu)$, the supremum of the concave function $\psi(x) = x \cdot y - \phi_{\mu}(x)$ occurs at a point $\tilde{x}$ for which $0 \in \partial \psi(\tilde{x})$. For $x_{k} \in F$, where $F \subset \text{int}(\Omega)$ is the set of points of differentiability of $\phi_{\mu}$, we have $\partial \psi(x_k) = \{ y - \nabla \phi_{\mu}(x_k) \}$. For $x_k$ converging to the boundary, we have that $\vert \nabla \phi_{\mu}(x_k) \vert \rightarrow \infty$ by Lemma~\ref{lemma:gradientexplosion}. By the continuity of subdifferentials, this applies to the subdifferential of $\psi$ as well. Hence, for $x$ near enough to the boundary, we cannot have $0 \in \partial \psi(x)$. Hence, $\tilde{x}$ such that $\phi_{\nu}(y) = \tilde{x} \cdot y - \phi_{\mu}(\tilde{x})$ must satisfy $\tilde{x} \in \text{int}(\Omega)$.


From this, we deduce, for every $y \in \mathbb{R}^d$:
\begin{equation}\label{eq:xtilde}
\phi_{\nu;R}(y) - \phi_{\nu}(y) \geq \phi_{\mu}(\tilde{x}) - \phi_{\mu;R}(\tilde{x}) \rightarrow 0,
\end{equation}
since by $\tilde{x} \in \text{int}(\Omega)$ we can apply the pointwise convergence in Theorem~\ref{thm:uniform}.


We now show that $\tilde{x}_{R}$ cannot march to the boundary. By definition, $y \in \partial \phi_{\mu;R}(\tilde{x}_{R})$. Suppose that $\tilde{x}_{R_{k}} \rightarrow x' \in \partial \Omega$. Fix $\delta_1>0$. We choose a sequence $\{ \tilde{\tilde{x}}_{R_{k}} \}_{k}$ where $\tilde{\tilde{x_{R_{k}}}} \rightarrow x'' \in \text{int}(\Omega)$ and $\vert x' - x'' \vert \leq \delta_1$. Since $x''$ is in the interior and we have pointwise convergence of $\phi_{\mu;R}$ to $\phi_{\mu}$ in the interior, using Rockafellar~\cite{rockafellar} Theorem 24.5 we get that for every $\epsilon>0$, there exists an index $k_0$ such that for all $k \geq k_0$, we have
\begin{equation}\label{eq:close}
\partial \phi_{\mu;R_{k}}(\tilde{x}_{R_{k}}) \subset \partial \phi_{\mu} (x'') + \epsilon B_{1}(0).
\end{equation}
By the continuity of subdifferentials, we can decrease the size of $\delta_1$ and consequently $\partial \phi_{\mu}(x'')$ in contained in a set of values arbitrarily close to $\nabla \phi_{\mu}(x''')$, where $x'''$ is a point of differentiability arbitrarily close to $x''$ and to the boundary $\partial \Omega$. Thus, by decreasing the size of $\delta_1$, $\partial \phi_{\mu}(x'')$ is taken arbitrarily far from the value $y$, since by continuity it must be arbitrarily close to $\nabla \phi_{\mu}(x''')$, which diverges in magnitude. But, also by Equation~\eqref{eq:close}, we have $y \in \partial \phi_{\mu} (x'') + \epsilon B_{1}(0)$. By decreasing the size of $\epsilon$, we see that $\partial \phi_{\mu}(x'')$ cannot be taken arbitrarily far from $y$ as $\delta_1$ decreases. Therefore, we have a contradiction and hence $\tilde{x}_{R}$ does not converge to the boundary.


For every convergent subsequence $\tilde{x}_{R_{k}}$, we get that $\tilde{x}_{R_{k}} \rightarrow x' \in \text{int}(\Omega)$, where $x'$ depends on the choice of subsequence, but is not on the boundary. Using this, we can show:
\begin{equation}\label{eq:xtilde2}
\phi_{\nu;R_{k}}(y) - \phi_{\nu}(y) \leq \phi_{\mu}(\tilde{x}_{R_{k}}) - \phi_{\mu;R_{k}}(\tilde{x}_{R_{k}}).
\end{equation}
By the continuity of $\phi_{\mu}$ and $\phi_{\mu;R}$ in $\text{int}(\Omega)$, for any $\epsilon>0$ we have for large enough $k$, that:
\begin{equation}\label{eq:xRtilde}
\phi_{\nu;R_{k}}(y) - \phi_{\nu}(y) \leq \phi_{\mu}(x') - \phi_{\mu;R_{k}}(x') + \epsilon.
\end{equation}
Taking the limit in $k$ and using Theorem~\ref{thm:uniform} since $x' \in B_{\tilde{R}}(0)$, we get:
\begin{equation}
\lim_{k \rightarrow 0} \phi_{\nu;R_{k}}(y) - \phi_{\nu}(y) \leq \epsilon.
\end{equation}
Since the choice of $\epsilon$ was arbitrary, we get:
\begin{equation}\label{eq:subseq}
\lim_{k \rightarrow 0} \phi_{\nu;R_{k}}(y) - \phi_{\nu}(y) \leq 0.
\end{equation}
Since every convergent subsequence $\{ \tilde{x}_{R_{k}} \}$ leads to Inequality~\eqref{eq:subseq}, we conclude that any sequence $\{ \tilde{x}_{R_{k}} \}$ leads to the conclusion in Inequality~\eqref{eq:subseq}. Suppose this were not true, i.e. there existed a sequence $\{ \tilde{x}_{R_{l}} \}$ such that for every $l_0 \in \mathbb{N}$, there exists an $l \geq l_0$ such that $\phi_{\nu;R_{l}}(y) > \phi_{\nu}(y)$. Form the subsequence $\{ \tilde{x}_{R_{l'}} \}$ where
\begin{equation}\label{eq:strict}
\phi_{\nu;R_{l'}}(y) > \phi_{\nu}(y),
\end{equation}
for all $l'$. This subsequence may not be convergent, but it has a convergent subsequence, which we call $\{ \tilde{x}_{R_{l''}} \}$. Following the reasoning above, this convergent subsequence must satisfy Inequality~\eqref{eq:subseq}, contradicting Inequality~\eqref{eq:strict}. Thus, we are forced to conclude that for any sequence $\{ \tilde{x}_{R_{k}} \}$ we have $\lim_{k \rightarrow \infty} \phi_{\nu;R_{k}}(y) \leq \phi_{\nu}(y)$. Therefore,
\begin{equation}
\lim_{R \rightarrow \infty} \phi_{\nu;R}(y) - \phi_{\nu}(y) \leq 0.
\end{equation}
Along with Inequality~\eqref{eq:xtilde}, we get:
\begin{equation}
\lim_{R \rightarrow \infty} \vert \phi_{\nu;R}(y) - \phi_{\nu}(y) \vert = 0.
\end{equation}

\end{proof}

For the general case $p \geq 2$, we do not expect convergence on all of $\Omega$. However, we do believe that rates of convergence can be established for some subsets of $\Omega$, notably subsets $\Omega '$ where $\text{dist}(\Omega ', \partial \Omega)>0$. If such rates have been established, see Remark~\ref{rmk:geq2rates}, then we can bootstrap these to get a rate for the inverse Brenier potentials for the case $p \geq 2$.

\begin{proposition}\label{prop:bootstrap}
Let $\vert \phi_{\mu;R}(x) - \phi_{\mu}(x) \vert \leq C(\Omega ')k(R)$ for all $x \in \Omega' \subset \Omega$ for all $\Omega ' \subset \Omega$ such that $\text{dist}(\Omega ', \partial \Omega)>0$. Then, there exists an $R_0>0$ and a constant $C = C(\Omega '(y))>0$ such that for all $R \geq R_0$ we have:
\begin{equation}
\vert \phi_{\nu;R}(y) - \phi_{\nu}(y) \vert \leq C(\Omega '(y))k(R),
\end{equation}
for all $y \in \text{spt}(\nu)$. As indicated, the constant $C(\Omega '(y))$ need not be uniform in $y$.
\end{proposition}

\begin{proof}
Like the proof of Theorem~\ref{thm:pleqninverse}, we simply need to find, for a given $y \in \text{spt}(\nu)$, a subset $\Omega '(y) \subset \Omega$ where for a given $y \in \text{spt}(\nu)$, the choice of $\tilde{x}, x' \in \text{int}(\Omega '(y))$ (for any such $\tilde{x}, x'$ where we have used the notation from Theorem~\ref{thm:pleqninverse}). We then use Inequality~\eqref{eq:xtilde} and choosing $R$ large enough allows us to use Inequality~\eqref{eq:xRtilde}. We thus assert that:
\begin{equation}
\vert \phi_{\nu;R}(y) - \phi_{\nu}(y) \vert \leq C(\Omega '(y))k(R),
\end{equation}
for $y \in \text{spt}(\nu)$. When $y$ is such that $\text{dist}(\Omega ', \Omega)$ is very small, we still have the same asymptotic rate, but the constant $C(\Omega '(y))$ blows up.

\end{proof}

\subsection{Convergence Almost Everywhere for Inverse Brenier Mapping}\label{sec:inversemap}

Provided that $\nu$ is absolutely continuous with respect to the Lebesgue measure, it can be readily shown that $T_{\nu;R}$ converges to $T_{\nu}$ in measure. That is, the probability (with respect to the measure $\nu$) of the set where $\nabla \phi_{\nu;R}$ and $\nabla \phi_{\nu}$ differ more than a fixed $\epsilon>0$ goes to zero, which can be done by following a well-known exercise in Villani~\cite{Villani1}. However, we actually have more information about the inverse Brenier mapping, since we know the inverse Brenier mapping is connected to the inverse Brenier potentials, which are convex. Pointwise convergence of the inverse Brenier potentials can be used to show pointwise almost-everywhere convergence for the Brenier mapping. This is shown using the property of stability of subdifferentials of convex functions under limits and by the fact that the Brenier mapping is contained within the subdifferential of the Brenier potentials at each point $x \in \text{spt}(\nu)$.

\begin{theorem}
Provided that $\nu$ is absolutely continuous with respect to the Lebesgue measure on $\mathbb{R}^n$, then for $p>n$, we have that $T_{\nu;R} \rightarrow T_{\nu}$ pointwise almost everywhere on $\text{int}(\text{spt}(\nu))$, where $T_{\nu}$ is any optimal mapping. For $p \leq n$, let $\text{spt}(\nu)$. If $\phi_{\nu;R}$ and $\phi_{\nu}$ satisfy $\phi_{\nu;R} \rightarrow \phi_{\nu}$ pointwise, then $T_{\nu;R} \rightarrow T_{\nu}$ pointwise almost everywhere on $\mathbb{R}^n$, where $T_{\nu}$ is any mapping satisfying $T_{\nu}(x) = \nabla \phi_{\nu}(x)$ for almost every $x \in \mathbb{R}^n$.
\end{theorem}
\begin{proof}
In both cases, since $\phi_{\nu;R} \rightarrow \phi_{\nu}$ pointwise on $\text{spt}(\nu)$, then by Rockafellar Theorem 24.5~\cite{rockafellar}, we have that $\lim_{R \rightarrow \infty} \partial \phi_{\nu;R}(x) \subset \partial \phi_{\nu}(x)$ on $\text{int}(\text{spt}(\nu))$, since $\phi_{\nu;R}(x)$ limits to a finite number everywhere on $\text{int}(\text{spt}(\nu))$. By the Brenier-McCann theorem and the Kantorovich Theorem~\cite{maggi}, the optimal inverse mappings satisfy $T_{\nu;R}(x) \in \partial \phi_{\nu;R}(x)$ for $x \in \text{spt}(\nu_{R})$ and $T_{\nu}(x) \in \partial \phi_{\mu}(x)$ for $x \in \text{spt}(\nu)$. Therefore, $T_{\nu;R}(x) \in \partial \phi_{\nu;R}(x) \rightarrow \{ \nabla \phi_{\nu}(x) \} = \partial \phi_{\nu}(x)$ for any $x \in F$, where $F \subset \text{spt}(\nu)$ denotes the set of differentiability points of $\phi_{\nu}$. Since $\phi_{\nu}$ is convex, by Radamacher's theorem $T_{\nu;R}(x) \rightarrow T_{\nu}(x)$ almost everywhere. Let $S$ and $T$ are optimal maps, then $S = T$ $\nu$-a.e. by the Brenier-McCann Theorem. Thus, since $\nu$ is absolutely continuous with respect to Lebesgue, then $S = T$ almost everywhere.
\end{proof}

\section{Convergent Numerical Method for Approximating Elliptic PDEs on Unbounded Domains}\label{sec:computation}
In this section, we briefly discuss how provably convergent PDE-based schemes can be designed to solve the optimal transport problem from a measure with bounded support to a measure with unbounded support in light of the convergence results obtained in Sections~\ref{sec:radial} and~\ref{sec:general}. We begin with an introduction to the notation and key concepts for provably convergent monotone schemes for such fully nonlinear elliptic PDE. We then propose to use and discuss the results from Hamfeldt~\cite{HamfeldtBVP2} that can be used, for a fixed $R$, to solve the problem. We discuss how such schemes can be adapted to the current problem, including a discussion on boundary conditions, whether it is better to solve the forward or inverse problem, interpolation, and mention some topics untouched by the current work.

\subsection{Background and Definitions}

First, we will considering fully nonlinear elliptic second-order elliptic PDEs:
\begin{equation}\label{eq:elliptic}
F(x, u(x), Du(x), D^2 u(x)) = 0,
\end{equation}
for $x \in \Omega$, where $\Omega$ is a closed, compact subset of Euclidean space. At this point, the operator $F$ may encode boundary conditions for $x \in \partial \Omega$.

We introduce the approximation schemes and properties of such schemes that will allows us to construct a convergence proof. First, for a fixed domain $\Omega$, we define a grid $\mathcal{G}^h \subset \Omega$, which is a finite set. The discretization parameter $h>0$ is defined as $h:= \sup_{x \in \mathcal{G}^h} \inf_{x, y \in \mathcal{G}^h} \vert x - y \vert$. The discretization parameter $h$ is used in the notation for the grid $\mathcal{G}^h$ with the intention of showing that grids are ``refined" as $h \rightarrow 0$. In all of our convergence results, we will be taking $h \rightarrow 0$, so we assume that our grid has sufficient regularity to guarantee this.

We define a grid function $u^h:\mathcal{G}^h \rightarrow \mathbb{R}$. This grid function will solve a system of $N$ nonlinear algebraic equations, where $N$ is the cardinality of $\mathcal{G}^h$. The equations they will solve will be denoted as:
\begin{equation}
F^h(x, u^h(x), u^h(x) - u^h(\cdot)) =0,
\end{equation}
where $x \in \mathcal{G}^h$ and $u(x) - u(\cdot)$ is shorthand for differences $u(x) - u(y)$, where $y \neq x$ and $y \in \mathcal{G}^h$. We will think of $F^h$ as an operator and often write $F^h(x, u(x), u(x) - u(\cdot))$ for some function $u$ with the understanding that even though $u$ is not a grid function, the discrete operator $F^h$ is only evaluated at points on the grid. The hope is that a grid function satisfying $F^h(x, u^h(x), u^h(x) - u^h(\cdot)) =0$ will satisfy $u^h \rightarrow u$ as $h \rightarrow 0$ in the sense that if we interpolated the grid function $u^h$ onto $\Omega$ to form the interpolated function $U^h$, then we have $U^h \rightarrow u$ pointwise on $\Omega$.

In order to achieve convergence, we need the approximation scheme $F^h$ to ``mimic" the PDE. Thus follows the usual definition of consistency:
\begin{definition}
The approximation operator $F^h$ is consistent if for any $u \in C^2(\Omega)$, we have
\begin{equation}
\liminf_{h \rightarrow 0, y \in \mathcal{G}^h \rightarrow x, \xi \rightarrow 0} F^h(y_i, u(y_i) + \xi, u(y_i) - u(y_j)) = F(x, u(x), \nabla u(x), D^2 u(x)),
\end{equation}
for $x \in \Omega$.
\end{definition}

Another important property of such approximation schemes is monotonicity:
\begin{definition}
The discrete operator $F^h(x, u(x), u(x) - u(\cdot))$ is monotone if it is a non-increasing function of its final argument.
\end{definition}

\begin{definition}
A solution $u^h$ of an approximation scheme $F^h$ is said to be stable if there exists a constant $M>0$ such that $\vert u^h \vert \leq M$ for all $h>0$.
\end{definition}

\begin{definition}
An approximation scheme $F^h(x, u(x), u(x) - u(\cdot))$ is said to be continuous if it is continuous with respect to its last two arguments.
\end{definition}

\subsection{Convergence Result of Hamfeldt~\cite{HamfeldtBVP2}}\label{sec:Hamfeldt}

Suppose we have measures $\mu_0$ and $\mu_1$ supported on sets $X$ and $Y$ which are bounded subsets of $\mathbb{R}^n$ with density functions $f_0$ and $f_1$, respectively. Then, formally, the solution of the optimal transport problem with squared distance cost function $c(x,y) = \vert x - y \vert^2$ can be characterized as the solution of the following equation:
\begin{align}\label{eq:MAPDE}
f_1(\nabla u(x)) \det(D^2 u(x)) &= f_0(x) \\
\nabla u(X) &\subset Y.
\end{align}
This is known as the second boundary value problem of the Monge-Amp\`{e}re PDE, see Urbas~\cite{Urbas}. The constraint $\nabla u(X) \subset Y$ can be rewritten using, for example a signed distance function $H$, which is a continuous function $H: \mathbb{R}^n \rightarrow \mathbb{R}$ satisfying $H(y) < 0$ for $y \in Y$, $H(y)= 0$ for $y \in \partial Y$ and $H(y)>0$ for $y \notin Y$. Thus, one can rewrite the global constraint $\nabla u(X) \subset Y$ as the boundary condition $H(\nabla u(x)) = 0$ for $x, \in \partial X$.

The PDE~\ref{eq:MAPDE} is elliptic on the space of convex functions. To restrict the numerical method to only select convex solutions, we define the operator $\lambda_1(D^2u(x)) = \min_{\theta \in \mathbb{S}^{n-1}} \frac{\partial^2 u(x)}{\partial \theta^2}$, that is the smallest eigenvalue of the Hessian $D^2u(x)$. We thus define the following modified PDE:

\begin{equation}
\max \{ -f_1(\nabla u(x)) \det (D^2 u(x)) + f_0(x), - \lambda_1 (D^2 u(x)), H(\nabla u(x)) = 0,
\end{equation}
a weak solution solution of such a PDE is expected to not only solve the Monge-Amp\`{e}re PDE, but also be convex and satisfy the global constraint $\nabla u(X) \subset Y$. The natural notion of a weak solution of such a PDE is the viscosity solution, see Crandall, Ishii, and Lions~\cite{CIL}, which in Hamfeldt~\cite{HamfeldtBVP2} is shown to be equivalent to a unique (up to an additive constant) convex Aleksandrov solution of the Monge-Amp\`{e}re PDE, and therefore is a potential function for the optimal transport problem.

\begin{definition}
We say that a convex function $u$ is an Aleksandrov solution of the Monge-Amp\`{e}re equation~\eqref{eq:MAPDE} if
\begin{equation}
\int_{E} f(x)dx = \int_{\partial u(E)} g(y)dy,
\end{equation}
for every measurable set $E \subset X$, where $\partial u(E)$ is the subdifferential of $u$ over the set $E$, see Definition~\ref{def:subdifferential}.
\end{definition}

\begin{hypothesis}
We impose the following hypotheses on the problem:
\begin{enumerate}
\item $X$, $Y$ are bounded and $Y$ is convex
\item $f_0 \in L^1(X)$ is lower semi-continuous
\item $f_1 \in L^1(Y)$ is positive on $Y$ and upper semicontinuous
\end{enumerate}
\end{hypothesis}

The first result of Hamfeldt~\cite{HamfeldtBVP2} is that certain monotone approximation schemes are guaranteed to have a stable solution. This uses the slightly modified definition of consistency:
\begin{definition}
The scheme $F^h$ is consistent on $X$ if for any smooth convex function $u$ and $x \in X$,
\begin{equation}
\liminf_{h \rightarrow 0, y \in \mathcal{G}^h \rightarrow x, \xi \rightarrow 0} F^h(y_i, u(y_i) + \xi, u(y_i) - u(y_j)) \geq F(x, u(x), \nabla u(x), D^2 u(x)).
\end{equation}
The scheme is uniformly consistent if it is consistent and the limit inferior is achieved uniformly on $X$.
\end{definition}

To guarantee stability and therefore convergence to weak solutions (viscosity solutions or Aleksandrov solutions) of fully nonlinear elliptic PDE, it has been found in Hamfeldt~\cite{HamfeldtBVP2} and Hamfeldt and Turnquist~\cite{HT_OTonSphere}, for example, that an important property of the scheme is that it is underestimating.
\begin{definition}
An approximation scheme $F^h$ is said to be underestimating if
\begin{equation}
F^h(x_i, u(x_i), u(x_i) - u(x_j)) \leq \max \{-f_1(\nabla u(x)) \det(D^2 u(x)) + f_0(x), -\lambda_1(D^2 u(x)), H(\nabla u(x)) \},
\end{equation}
for every smooth convex function $u$ and $x \in \mathcal{G}^h \cap X$.
\end{definition}
Note that it is, in practice, not difficult to obtain an underestimating scheme, since we can take any consistent, monotone scheme $F^h$ and then replace it with $F^h - h^{\alpha}$, where $h^{\alpha}$ is related to the formal consistency error of the scheme $F^h$.

\begin{theorem}
Let $F^h$ be a uniformly consistent, monotone, continuous, strictly underestimating approximation scheme. Then, the approximation scheme has a solution $u^h$ and the solution is stable.
\end{theorem}

The main result is Theorem 29 from Hamfeldt~\cite{HamfeldtBVP2}, which states that:
\begin{theorem}
Let $u^h$ be any solution of the consistent, monotone, stable approximation scheme. Defining the piecewise constant nearest-neighbors extension:
\begin{equation}
U^h(x) = \sup \left\{ u^h(y) \vert y \in \mathcal{G}^h, \vert y - x\vert = \min_{z \in \mathcal{G}^h} \vert z - x \vert \right\},
\end{equation}
then, for any $x \in X$, $\lim_{h \rightarrow 0}$, where $u(x)$ is an Aleksandrov solution of Equation~\eqref{eq:MAPDE}.
\end{theorem}

What this then implies is that any consistent, monotone, and stable approximation $F^h$ can be used, for example one whose formal consistency error rapidly decreases to zero. In Hamfeldt~\cite{HamfeldtBVP2} a method with a small formal consistency error is proposed.

\subsection{Discussion of Convergent Schemes for Optimal Transport on Unbounded Domains}

Section~\ref{sec:radial} we saw that in the radial case that we had nice convergence properties for a wide variety of density functions and for a wide variety of cost functions. Thus, it appears plausible that convergent schemes may work well for such cost functions, although the proper justification for this is beyond the scope of this paper. For the cost function $c(x,y) = \vert x - y \vert^2$, in Section~\ref{sec:general}, we established many results on the convergence as $R \rightarrow \infty$. The strongest of such results yielded explicit (not asymptotic) rates for target density functions that satisfied $p > n$ and $p>4$, where $p$ is the highest moment of $\nu$. For density functions that are log-concave, the convergence rate was exponential. Thus, the convergence is well-justified in some cases for the squared distance cost function.

For the squared cost function, we therefore propose fixing a ``large" value of $R$, and choosing to solve the following optimal transport problem using the techniques from Hamfeldt~\cite{HamfeldtBVP2}. Let $\mu$ with density function $g$ have support on a bounded set $\Omega \subset \mathbb{R}^n$ and let $\nu$ with density function $f$ have support on $\mathbb{R}^n$. Define the density function:
\begin{equation}
f_{R_{C}}(x):=
\begin{cases}
f(x)/\nu(RC_{1}(0)), & \text{for} \ x \in RC_{1}(0) \\
0, &\text{otherwise}.
\end{cases}
\end{equation}
Let $\nu_{R_{C}}$ be the probability measure on $\mathbb{R}^n$ with density function $f_{R_{C}}$.

If one wants to solve for $\phi$, the Brenier potential from a source measure with bounded support to a target measure with unbounded support, the question is whether it is better to solve the forward problem or the inverse problem and then interpolate. As established in Section~\ref{sec:general}, in the best case the convergence is the same. For the forward problem, in cases where the potential function $\phi$ is not known to be bounded for the general case $p>2$, we may have issues for large $R$ related to this blowup, or to the fact that the target density function $f$ is not strictly bounded away from zero on $\mathbb{R}^n$. Thus, as $R \rightarrow \infty$, there may be stability issues that arise. These issues will investigated in further numerical work. Thus, we may be able to also solve the inverse problem and then interpolate.

Thus, in order to solve the optimal transport problem from $\mu$ to $\nu$, we propose to numerically compute the solution of the following Monge-Amp\`{e}re PDE:
\begin{align}
f_{R_{C}} (\nabla u(x)) \det (D^2 u(x)) &= g(x) \\
\nabla u(RC_{1}(0)) &\subset \Omega.
\end{align}

As explained in Section~\ref{sec:Hamfeldt}, this requires the construction of a uniformly consistent, monotone, stable, continuous, and strictly underestimating scheme $F^h$ of the operator $F(x,u(x), Du(x), D^2 u(x)):= -f_{R_{C}} (\nabla u(x)) \det (D^2 u(x)) +g(x)$, the construction of a signed distance function $H$ for $\Omega$, and an approximation of the smallest eigenvalue of the Hessian. Details can be found in Hamfeldt~\cite{HamfeldtBVP2} where the author proposes a discretization with a small formal consistency error. Conceivably, a very quickly convergent scheme like the spectral Galerkin method proposed in Jin et al~\cite{spectralgalerkin} could be used since the computational domain is rectangular.

\begin{remark}
The convergence framework of Hamfeldt~\cite{HamfeldtBVP2} does not allow for the domain $\Omega$ to change in size. Thus, the current framework does not allow for the parameter $R$, to change, for example, as a function of $h$. This avenue of approach will be investigated in further work.
\end{remark}

\section{Conclusion}\label{sec:conclusion}
In this manuscript, we have established many results on the pointwise convergence for the optimal transport problem from a source measure with bounded support to a target measure with unbounded support by using a cutoff approximation of the latter problem. The results are for quite general cost functions in the radially symmetric case for the cutoff approximation. For the cost function $c(x,y) = \vert x - y \vert^2$, in the general case, building off previous results by Delalande and M\`{e}rigot, we saw rapid convergence in some cases for the cutoff approximation. The main result of this manuscript show that for the general case, when is $\mu$ an $\nu$ have density functions, $p>n$ and $p>4$, the source domain $\Omega$ is convex, the density function of the source measure is bounded away from zero and infinity that we have pointwise convergence of the Brenier potentials and almost-everywhere convergence of the optimal transport mapping, only assuming that $\mu_{R} \rightharpoonup \nu$ weakly. Further refining the geometrical assumptions on the domain and exploiting the H\"{o}lder continuity property of the Brenier potentials we get explicit convergence rates depending on the explicit rate of convergence of the quantity $W_1(\nu_{R}, \nu)$. This quantity could then be explicitly found in the case of the cutoff approximation (and can, of course, be explicitly found in other cases as well). The convergence of the cutoff approximation was shown to be quite fast, in fact exponentially fast in the case of $log$-concave measures. We then derived many other convergence guarantees slightly lifting the assumptions on the moments of the target distribution, in anticipation that such results will be derived in the future. We finally derived convergence guarantees for the inverse Brenier potentials and the inverse optimal transport mapping. We finished with a discussion of one particular provably convergent numerical method that can be used to solve the resulting cutoff approximation problem.

\bigskip

\textbf{Acknowledgement.} The author acknowledges the support of the Beijing Natural Science Foundation BJNSF ISP.

\bibliographystyle{plain}
\bibliography{ThreeSystems}

\end{document}